\documentclass[10pt]{amsart}
\usepackage[english]{babel}
\usepackage[pdfborder={0 0 0}]{hyperref}
\usepackage{amsmath, amssymb, amsthm}
\usepackage{enumitem}
\usepackage{tikz}
\usepackage{colonequals}
\usepackage{fullpage}
\usepackage{bbm}

\theoremstyle{plain}                            
\newtheorem{stelling}{Theorem}[section]
\newtheorem{gevolg}[stelling]{Corollary}
\newtheorem{lemma}[stelling]{Lemma}
\newtheorem{prop}[stelling]{Proposition}

\DeclareMathOperator{\bfk}{\boldsymbol{k}}

\theoremstyle{definition}
\newtheorem{definition}[stelling]{Definition}
\newtheorem{opm}[stelling]{Remark}
\newtheorem{voorbeeld}[stelling]{Example}
\newenvironment{vb}
  {\pushQED{\qed}\voorbeeld}
  {\popQED\endvoorbeeld}

\title{Congruence conditions for the mod $\boldsymbol{\lambda}$ values of the Fourier coefficients of classical eigenforms}

\date{\today}

\author{Michael A. Daas}

\begin{document}

\begin{abstract}
We classify all instances of the condition $a_{p}(f) \equiv x \bmod \lambda$ being related to a congruence on the prime $p$, where $a_{p}(f)$ denotes the $p$th Fourier coefficient of a classical normalised cuspidal eigenform $f$ and $\lambda$ is a prime in the number field generated by the Fourier coefficients of $f$. This classification is done in terms of the (projective) image of the mod $\lambda$ Galois representation associated with $f$ and extends work by Swinnerton-Dyer. We highlight that for $x = 0$, this condition is more often implied by a congruence on the prime $p$ than the general value of $a_{p}(f) \bmod \lambda$. Finally, we illustrate various instances of these congruences through examples from the setting of weight 2 newforms attached to rational elliptic curves.
\end{abstract}

\maketitle

\setcounter{tocdepth}{1}
\tableofcontents

\section{Introduction}\label{intro}

Building on the work of Ramanujan, congruences between modular forms enjoy a long and rich history, and informed the seminal work of Serre \cite{serre} and his notion of $\ell$-adic modular forms, as $\ell$-adic limits of $\textbf{q}$-expansions coming from classical modular forms satisfying congruences modulo increasingly high powers of $\ell$. The Fourier coefficients of modular forms, especially normalised eigenforms, often contain a wealth of arithmetic information. Given such a modular cuspform $f$ with Fourier coefficients $a_n(f)$ for integers $n \geq 1$ and a prime $\lambda$ in the number field generated by the Fourier coefficients of $f$, one may ask the question of classifying all primes $p$ for which $a_{p}(f) \equiv x \bmod \lambda$ for a given residue class $x \bmod \lambda$.

In general, this condition can be rather intractible. Even though it is governed by the behaviour of Frobenius at $p$ in a certain field extension, the general non-abelian nature of this extension often prevents one from writing down simple conditions that classify these primes $p$ succinctly. Instead, they are often subject to non-abelian reciprocity laws, as first described in \cite{shimura}. 

For rational normalised eigenforms for the full modular group, all possible congruences for their Fourier coefficients were classified by Swinnerton-Dyer in Lemma 1 and the subsequent Corollary in both \cite{SD1} and \cite{SD2}. The present work aims to extend and refine this analysis to a larger set of classical modular forms, and to illustrate how level structure and the field of definition affect the analysis. It also fulfils an expository purpose by providing elementary proofs where possible and ample examples to support the main results. 

\subsection{Setup and definitions}

\!We recall some notation. Let $\bfk \geq 2$ and $N \geq 1$ be integers and let $f \in \mathcal{S}_{\bfk}( N, \chi )$ be a weight $\bfk$ normalised cuspidal eigenform of level $N$ with nebentypus $\chi : ( \mathbb{Z} / N \mathbb{Z} )^{\times} \to \mathbb{C}^{\times}$. Let 
\[
K_f \colonequals \mathbb{Q}\big( \{ a_n(f) \mid n \in \mathbb{N} \} \big)
\]
be the number field generated by the Fourier coefficients of $f$ and let $\lambda$ be a prime of the ring of integers $\mathcal{O}_f$ of $K_f$. We let $\mathbb{F}_{\lambda} \colonequals \mathcal{O}_f / \lambda$ denote the residue field of $\lambda$ and let $\ell \colonequals \text{char}( \mathbb{F}_{\lambda} )$ denote its characteristic.

By a famous theorem of Deligne \cite{deligne}, attached to $f$ there is an irreducible 2-dimensional $\lambda$-adic Galois representation
\[
\rho_{f, \lambda} : G_{\mathbb{Q}} \to \text{GL}_2( K_{f, \lambda} ),
\]
where $G_{\mathbb{Q}} \colonequals \text{Gal}( \overline{ \mathbb{Q} } / \mathbb{Q} )$ and where $K_{f,\lambda}$ denotes the completion of $K_f$ at the prime $\lambda$.  This representation is unramified at all primes $p \nmid N \ell$ and for any prime $\mathfrak{p} \subset \overline{ \mathbb{Z} }$ lying over such $p$, the characteristic polynomial of $\rho_{f, \lambda}( \text{Frob}_{\mathfrak{p}} )$ is given by
\[
 X^2 - a_p(f)X+\chi(p)p^{\bfk-1},
\]
where $X$ is an indeterminate. The representation is also \emph{odd} in the sense that the determinant of the image of complex conjugation is $-1$. Since $G_{\mathbb{Q}}$ is compact as a topological group and $\rho_{f, \lambda}$ is continuous, this representation space admits a Galois stable lattice. Any choice of such a lattice will induce an integral Galois representation that we will, by slight abuse of notation, also denote by
\[
\rho_{f, \lambda} : G_{\mathbb{Q}} \to \text{GL}_2( \mathcal{O}_{f, \lambda} ),
\]
where $\mathcal{O}_{f, \lambda}$ denotes the ring of integers in $K_{f,\lambda}$. We note that different choices of such stable lattices need not induce isomorphic Galois representations in this integral context, but their \emph{semisimplifications}, denoted $\rho^{\text{ss}}_{f, \lambda}$, will always be isomorphic. Composition with the natural reduction map $\mathcal{O}_{f, \lambda} \to \mathbb{F}_{\lambda}$ yields a Galois representation which we will denote by
\[
\rho_f^{\lambda} : G_{\mathbb{Q}} \to \text{GL}_2( \mathbb{F}_{\lambda} ),
\]
and $\rho_f^{\text{ss}, \lambda}$ will denote its semisimplification. We also use the notation $\overline{\rho}_f^{\lambda}$ for the composition of $\rho_f^{\lambda}$ with the natural projection map $\text{GL}_2( \mathbb{F}_{\lambda} ) \to \text{PGL}_2( \mathbb{F}_{\lambda} )$. Finally, we let $G_{\lambda}$ and $\overline{G}_{\lambda}$ denote the images of $\rho_f^{\lambda}$ and $\overline{\rho}_f^{\lambda}$ respectively, dropping $f$ from the notation to improve readability.

In recent work, Charlton, Medvedovsky and Moree \cite{MoreeQ} studied the Euler-Kronecker constants of Dirichlet L-functions associated with multiplicative sets (i.e. a subset of the positive integers $\mathbb{N}$ whose indicator function is multiplicative) of the form $\{ n \in \mathbb{N} : \ell \nmid a_n(f) \}$ for various classical modular eigenforms $f$. These values were computed by relating these L-functions to certain Artin L-functions, whose Euler-Kronecker constants are more easily computed to comparatively high precision. 

It was further observed that, even in cases when the general behaviour of $a_{p}(f) \bmod \lambda$ is governed by a non-abelian extension, it can happen that the \emph{vanishing} of this quantity is given by a congruence condition on the prime $p$, simplifying their computations. This motivates the following definitions.

\begin{definition}\label{fproper}
A class $x \in \mathbb{F}_{\lambda}$ is called $f$-\emph{proper} if there exists a prime $p \nmid N \ell$ such that $a_p( f ) \equiv x \bmod \lambda$.
\end{definition}

\begin{definition}\label{lab}
Let $f \in \mathcal{S}_{\bfk}( N, \chi )$ be normalised cuspidal eigenform and let $x \in \mathbb{F}_{\lambda}$ be an $f$-proper class. 
\begin{itemize}
\item We say that $f$ is \emph{weakly $\lambda$-abelian} for the class $x \in \mathbb{F}_{\lambda}$ if there exist a positive integer $M$ and a non-empty subset $\varnothing \neq S_x \subset (\mathbb{Z} / M \mathbb{Z})^{\times}$ such that for all $p \nmid N \ell$,
\[
p \bmod M \in S_x \implies a_p( f ) \equiv x \bmod \lambda.
\]
\item We say that $f$ is \emph{semi-$\lambda$-abelian} for the class $x \in \mathbb{F}_{\lambda}$ if there exist a positive integer $M$ and a proper subset $S_x \subsetneq (\mathbb{Z} / M \mathbb{Z})^{\times}$ such that for all $p \nmid N \ell$,
\[
a_p( f ) \equiv x \bmod \lambda \implies p \bmod M \in S_x.
\]
\item We say that $f$ is \emph{$\lambda$-abelian} for the class $x \in \mathbb{F}_{\lambda}$ if it is both weakly $\lambda$-abelian and semi-$\lambda$-abelian for $x \in \mathbb{F}_{\lambda}$. Equivalently, if there exist a positive integer $M$ and a subset $S_x \subset (\mathbb{Z} / M \mathbb{Z})^{\times}$ such that for all $p \nmid N \ell$, \vspace{-2mm}
\[
a_p( f ) \equiv x \bmod \lambda \iff p \bmod M \in S_x.
\]
\item Finally, we say that $f$ is \emph{totally $\lambda$-abelian} if it is $\lambda$-abelian for all $f$-proper $x \in \mathbb{F}_{\lambda}$. 
\end{itemize}
\end{definition}

\begin{vb}\label{introexam}
The weight 12 Ramanujan-$\Delta$-function is defined as
\[
\Delta( \textbf{q} ) = \eta( \textbf{q} )^{24}, \quad \text{where} \quad \eta( \textbf{q} ) = \textbf{q}^{1/24} \prod_{n=1}^{\infty} ( 1 - \textbf{q}^n ).
\]
It is now easy to see that $\Delta( \textbf{q} ) \equiv \eta( \textbf{q} ) \eta( \textbf{q}^{23} ) \bmod 23$, which establishes a congruence between $\Delta$ and a cusp form of weight 1. This observation can be used to explain that the mod 23 Galois representation $\rho_{\Delta}^{23}$ associated with $\Delta \in \mathcal{S}_{12}( \text{SL}_2( \mathbb{Z} ) )$ is in fact given by an irreducible Artin representation whose kernel corresponds to the Hilbert class field of $\mathbb{Q}( \sqrt{-23} )$, which has Galois group $S_3$ over $\mathbb{Q}$. One deduces that
\[
a_p( \Delta ) \bmod 23 = \begin{cases} 0 \bmod 23 &\text{if $-23$ is not a square modulo $p$;} \\ 2 \bmod 23 &\text{if $p = x^2 + 23y^2$ for some $x,y \in \mathbb{Z}$;} \\ -1 \bmod 23 &\text{otherwise.} \end{cases}
\]
The condition $a_p( \Delta ) \equiv 0 \bmod 23$ is thus a congruence condition on $p$, whereas for the two other $\Delta$-proper residue classes mod 23 this is not true. In other words, the modular form $\Delta$ is $23$-abelian for the class $0 \bmod 23$, but it is not totally 23-abelian. For more context, see Example 1.2 in \cite{vonk2021}.
\end{vb}

\subsection{The classifications}

This paper determines necessary and sufficient conditions for all the notions introduced in Definition \ref{lab}. These conditions will be expressed in terms of the images of the associated mod $\lambda$ Galois representations $\rho_f^{\lambda}$ and $\overline{\rho}_f^{\lambda}$. We let $\mathbb{F}_{\lambda^2}$ denote the unique quadratic extension of $\mathbb{F}_{\lambda}$. Recall that $G_{\lambda}$ and $\overline{G}_{\lambda}$ denote the images of $\rho_f^{\lambda}$ and $\overline{\rho}_f^{\lambda}$ respectively. Our first set of main results is now as follows.

\begin{stelling}\label{main1}
A normalised cuspidal eigenform $f \in \mathcal{S}_{\bfk}( N, \chi )$ is totally $\lambda$-abelian if and only if $G_{\lambda}$ is conjugate to a subgroup of the Borel subgroup of upper triangular matrices in $\emph{GL}_2( \mathbb{F}_{\lambda^2} )$.
\end{stelling}

\begin{stelling}\label{main2}
If a normalised cuspidal eigenform $f \in \mathcal{S}_{\bfk}( N, \chi )$ is weakly $\lambda$-abelian for some $f$-proper class $x \in \mathbb{F}_{\lambda}$, but \emph{not} totally $\lambda$-abelian, then $x = 0$. This happens if and only if $\overline{G}_{\lambda} \cong D_n$ is a dihedral group, where $n > 1$ and $\ell \nmid n$.
\end{stelling}

\begin{stelling}\label{main3}
If a normalised cuspidal eigenform $f \in \mathcal{S}_{\bfk}( N, \chi )$ is $\lambda$-abelian for some $f$-proper class $x \in \mathbb{F}_{\lambda}$ but \emph{not} totally $\lambda$-abelian, then $x = 0$. This happens if and only if $\overline{G}_{\lambda} \cong D_n$ is a dihedral group, where $n = 2$ or $n \geq 3$ is odd, and where $\ell$ is odd and $\ell \nmid n$.
\end{stelling}

\begin{opm}
These results reflect the content of the Corollary following Lemma 1 in \cite{SD2}. In case the image $G_{\lambda}$ of $\rho_f^{\lambda}$ is contained in the Borel subgroup, if $N = 1$ and $K_f = \mathbb{Q}$, Swinnerton-Dyer showed that in fact, $a_p(f) \equiv p^{m} + p^{k-1-m} \bmod \ell$ for some integer $0 \leq m \leq k-1$. The introduction of level structure prevents us from drawing a similar conclusion in this generality, as illustrated by Example \ref{338exam}. In the dihedral case, Swinnerton-Dyer showed that $a_p(f) \equiv 0 \bmod \ell$ if and only if $p$ is not a square modulo $\ell$. We recover a more general version of this in Proposition \ref{goodmod} below.
\end{opm}

We will prove the above results through elementary arguments. In particular, their proofs will not rely on Dickson's classification of the possible subgroups of $\text{PGL}_2( \mathbb{F}_{\lambda} )$; see Subsection \ref{dicksonsec} for a brief discussion. 

From Ribet's work on the Galois images of non-CM modular forms \cite{ribet}, it follows that a given eigenform without complex multiplication can only be (weakly) $\lambda$-abelian for a finite number of different primes $\lambda \subset \mathcal{O}_f$. In case it is, however, we can bound the size of the modulus that appears in Definition \ref{lab}. For a set of rational primes $S$ and an integer $n$, let $\text{gcd}(n,S)$ denote the part of $n$ supported at the primes in $S$. In addition, we let $\text{exp}(G) = \text{min} \{ n \in \mathbb{N} \mid g^n = 1 \text{ for all } g \in G \}$ denote the exponent of a finite group $G$, which is a divisor of $\# G$. We then have the following.

\begin{prop}\label{mod}
Let $f \in \mathcal{S}_{\bfk}( N, \chi )$ be a normalised eigenform and let $\lambda \subset \mathcal{O}_f$ be a prime. Let $S$ denote the set of all primes dividing $N \ell$. Then:
\begin{itemize}
\item If $f$ is totally $\lambda$-abelian, then the value of $a_p(f) \bmod \lambda$ is fully determined by the value of $p \bmod M$, where $M$ can be taken to be a positive divisor of $\emph{rad}( N \ell ) \cdot \emph{gcd}\big( 2 \emph{ exp}( \emph{im}( \rho_f^{\emph{ss}, \lambda} ) ), S \big)$.
\item If $f$ is weakly $\lambda$-abelian for $0 \in \mathbb{F}_{\lambda}$ because $\overline{G}_{\lambda} \cong D_n$ is a dihedral group for some $n > 1$ with $\ell \nmid n$, then there exists a divisor $M \in \mathbb{Z}$ of $\emph{rad}( N \ell ) \cdot \emph{gcd}( 2, N \ell )^2$ such that for $p \nmid N \ell$,
\[
\left( \frac{M}{p} \right) = -1 \implies a_p( f ) \equiv 0 \bmod \lambda.
\]
If $N\ell$ is odd, then $M \equiv 1 \bmod 4$. If $\ell \neq 2$ and $n \geq 3$ is odd, then the above is even an equivalence.
\item If $f$ is $\lambda$-abelian for $0 \in \mathbb{F}_{\lambda}$ because $\overline{G}_{\lambda} \cong D_2 = C_2 \times C_2$ and $\ell$ is odd, then there exist divisors $M_1, M_2 \in \mathbb{Z}$ of $\emph{rad}( N \ell ) \cdot \emph{gcd}( 2, N )^2$ such that for $p \nmid N \ell$,
\[
a_p( f ) \equiv 0 \bmod \lambda \iff \left( \frac{M_1}{p} \right) = -1 \quad \text{or} \quad \left( \frac{M_2}{p} \right) = -1.
\]
\end{itemize}
\end{prop}

In a common special case, we can be more precise.
\begin{prop}\label{goodmod}
Let $f \in \mathcal{S}_{\bfk}( \Gamma_0(N) )$ be a normalised eigenform and let $\lambda \subset \mathcal{O}_f$ be a prime with residue field $\mathbb{F}_{\lambda}$ of odd characteristic $\ell$. Suppose that $\bfk$ is even and that $[ \mathbb{F}_{\lambda} : \mathbb{F}_{\ell} ]$ is odd. If $\overline{G}_{\lambda} \cong D_n$ is a dihedral group for some odd integer $n$, then for $p \nmid N\ell$,
\[
a_p(f) \equiv 0 \bmod \lambda \iff \left( \frac{p}{\ell} \right) = -1.
\]
\end{prop}

This result recovers the observations from Example \ref{introexam} and the result by Swinnerton-Dyer in \cite{SD2}. In addition, other than the image of $\overline{\rho}_f^{\lambda}$ being a dihedral group $D_n$ with $n$ odd, all conditions from Proposition \ref{goodmod} are automatically satisfied in the case of rational newforms of weight 2 for $\Gamma_0(N)$, allowing it to be applied to the setting of modular forms associated with rational elliptic curves through Wiles's modularity theorem, which we will discuss in Section \ref{exam}. In addition, we will explain in Example \ref{counter} how this statement may fail in case $n$ is even. This does not contradict the result from \cite{SD2}, in whose setting the non-existence of a $C_2 \times C_2$-extension of $\mathbb{Q}$ ramified only at one odd prime $\ell$ prevents this case from occurring outright. 

Finally, the classification of all instances of being semi-$\lambda$-abelian cannot be expressed as succinctly as Theorems \ref{main1}, \ref{main2} and \ref{main3}, and makes implicit reference to Dickson's result.

\begin{stelling}\label{main4}
Let $f \in \mathcal{S}_{\bfk}( N, \chi )$ be a normalised cuspidal eigenform and let $c_{\lambda} \in [0,1]$ denote the density of primes $p$ for which $a_p(f) \equiv 0 \bmod \lambda$. The table below describes for all the possible projective images $\overline{G}_{\lambda}$ of $\overline{\rho}_f^{\lambda}$ whether $f$ is semi-$\lambda$-abelian for all $f$-proper classes in $\mathbb{F}_{\lambda}^{\times}$, whether $f$ is semi-$\lambda$-abelian for $0 \in \mathbb{F}_{\lambda}$, and the value of $c_{\lambda}$. We make the following remarks: \\
$-$ If $\overline{G}_{\lambda} \cong \emph{PSL}_2( k )$ for a field $k \subset \mathbb{F}_{\lambda}$, we assume $\ell \neq 2$. We write $\# k = q$, also when $\overline{G}_{\lambda} \cong \emph{PGL}_2( k )$. \\
$-$ In the Borel case, we also write $\# \mathbb{F}_{\lambda} = q$ and if $\ell \neq 2$, then $d$ must be even. \\
$-$ In case of $A_4$ or $S_4$ with $\emph{det}( G_{\lambda} ) = \mathbb{F}_3^{\times}$, the exceptions are precisely $\{ \pm 1 \} = \mathbb{F}_3^{\times}$. \\
$-$ Finally, condition $( \star )$ is the following:
\begin{itemize}
\item $\ell = 2$ and $\emph{det}( G_{\lambda} ) \subset \mathbb{F}_4^{\times}$; then $f$ is not semi-$\lambda$-abelian for any $x \in \mathbb{F}_4^{\times}$;
\item $\ell = 3$ and $\emph{det}( G_{\lambda} ) = \{ \pm 1 \}$; then $f$ is not semi-$\lambda$-abelian for four classes $x \in \mathbb{F}_9^{\times} \setminus \mathbb{F}_3^{\times}$; 
\item $\ell = 5$ and $\emph{det}( G_{\lambda} ) = \{ \pm  1\}$; then $f$ is not semi-$\lambda$-abelian for any $x \in \mathbb{F}_5^{\times}$;
\item $\ell = 29$ and $\emph{det}( G_{\lambda} ) = \{ \pm 1 \}$; then $f$ is not semi-$\lambda$-abelian for $x \in \{ \pm 2, \pm 5 \} \subset \mathbb{F}_{29}^{\times}$.
\end{itemize}
\end{stelling}
\renewcommand{\arraystretch}{1.25}
\begin{table}[h]
\centering
\begin{tabular}{|c||c|c|c|}
\hline
Projective image $\overline{G}_{\lambda}$ & Semi-$\lambda$-abelian on $\mathbb{F}_{\lambda}^{\times}$ & Semi for $0$ & $c_{\lambda}$ \\ \hline \hline
$\text{PGL}_2( \mathbb{F}_{\lambda} )$ and $\# \mathbb{F}_{\lambda} > 3$ & Not for any class & No & $q/(q-1)(q+1)$ \\ \hline
$\text{PSL}_2( \mathbb{F}_{\lambda} )$ and $\# \mathbb{F}_{\lambda} > 3$ & Not for any class & No & $1/(q + \epsilon )$ where $\epsilon = (-1)^{\frac{q+1}{2}}$ \\ \hline
$\text{PGL}_2( k )$ for $k \subsetneq \mathbb{F}_{\lambda}$ & Yes if $\text{det}( G_{\lambda} ) \not\subset k^{\times}$ & No & $q/(q-1)(q+1)$ \\ \hline
$\text{PSL}_2( k )$ for $k \subsetneq \mathbb{F}_{\lambda}$ & Yes if $\text{det}( G_{\lambda} ) \not\subset k^{\times}$ & No & $1/(q + \epsilon )$ where $\epsilon = (-1)^{\frac{q+1}{2}}$ \\ \hline
Contained in Borel & Yes & Yes & $0$ or $1/d$ for $d \mid q \pm 1$ \\ \hline
$D_n$ for $n$ odd and $\ell \nmid n$ & Yes & Yes if $\ell \neq 2$ & $1/2$ if $\ell \neq 2$, else $1/2 + 1/2n$ \\ \hline
$D_n$ for $n$ even and $\ell \nmid n$ & Yes & Yes & $1/2 + 1/2n$ \\ \hline
Isomorphic to $A_4$ & Yes unless $\text{det}( G_{\lambda} ) = \mathbb{F}_3^{\times}$ & Yes & $1/4$ if $\ell \neq 2$, else $1/3$ \\ \hline
Isomorphic to $S_4$ & Yes unless $\text{det}( G_{\lambda} ) = \mathbb{F}_3^{\times}$ & No & $3/8$ if $\ell \neq 2$, else $5/12$ \\ \hline
Isomorphic to $A_5$ & Yes unless $(\star)$ & No & $1/4$ if $\ell \neq 2$, else $4/15$ \\ \hline
\end{tabular}
\end{table}
\renewcommand{\arraystretch}{1}

We will illustrate the above results in Section \ref{exam}. 

\begin{opm}
The proofs below will rely on various occasions on explicit arithmetic with $2 \times 2$ matrices. At times, similar considerations were needed in the work of Faber in \cite{Faber}, or even in the classical work of Dickson and various other authors. The partly expository nature of this paper motivates us to give elementary proofs whenever this is possible, and to avoid Dickson's theorem unless it is really needed.

To stress the distinction between a result that is well known, and a result that contains more original work due to its focus on the specific interplay between commutator subgroups and the trace map, we have consistently labeled the former kind as a \emph{Lemma} in the following section, and the latter as a \emph{Proposition}. 
\end{opm}

\section{Proofs}\label{proof1}

We study the Galois representation $\rho_f^{\lambda} : G_{\mathbb{Q}} \to \text{GL}_2( \mathbb{F}_{\lambda} )$. Its kernel corresponds through Galois theory to some finite Galois extension $L_{\lambda}$ of $\mathbb{Q}$ which is ramified only at primes dividing $N \ell$. Its Galois group $\text{Gal}( L_{\lambda} / \mathbb{Q} )$ can be identified with the image $G_{\lambda}$ of $\rho_f^{\lambda}$. Generally the extension $L_{\lambda} / \mathbb{Q}$ is non-abelian, but it contains a maximal abelian extension, which we will denote by $Z_{\lambda}$. Its Galois group over $\mathbb{Q}$ can be identified with $G_{\lambda} / [ G_{\lambda}, G_{\lambda} ]$, where $[G_{\lambda}, G_{\lambda} ]$ denotes the commutator (or derived) subgroup of $G_{\lambda}$. 

The representation $\rho_f^{\lambda}$ identifies $\text{Gal}( L_{\lambda} / \mathbb{Q} ) \cong G_{\lambda}$ with its image inside $\text{GL}_2( \mathbb{F}_{\lambda})$, and equips it with a trace map coming from the usual
\[
\text{tr} : \text{GL}_2( \mathbb{F}_{\lambda}) \to \mathbb{F}_{\lambda},
\]
which is independent of the choice of basis. This trace map is a homomorphism when extended to the group $\text{Mat}_2( \mathbb{F}_{\lambda} )$ of $2\times2$ matrices over $\mathbb{F}_{\lambda}$ with addition, but on the multiplicative group $\text{GL}_2( \mathbb{F}_{\lambda} )$ it is somewhat ill-behaved, as it partitions this group into $\# \mathbb{F}_{\lambda}$ subsets that possess no clear multiplicative structure.

\begin{prop}\label{red1}
Let $f \in \mathcal{S}_{\bfk}( N, \chi )$ be a normalised eigenform and let $\lambda \subset \mathcal{O}_f$ be a prime. Then $f$ is weakly $\lambda$-abelian for some $x \in \mathbb{F}_{\lambda}$ if and only if there exists a coset of $[G_{\lambda}, G_{\lambda}]$ in $G_{\lambda}$ on which the trace is constantly equal to $x$.
\end{prop}
\begin{proof}
From Deligne's theorem, it is an immediate consequence that
\[
a_p(f) \equiv x \bmod \lambda \iff \text{tr}\big( \rho_{f, \lambda}( \text{Frob}_{\mathfrak{p}} ) \big) \equiv x \bmod \lambda \iff \text{tr}\left( \rho_f^{\lambda}( \text{Frob}_{\mathfrak{p}} ) \right) = x \in \mathbb{F}_{\lambda},
\]
where $\mathfrak{p} \subset \overline{ \mathbb{Z} }$ denotes a prime lying over a rational prime $p \nmid N\ell$. By the Kronecker-Weber theorem, the class of the Frobenius elements associated with unramified primes $p$ in the abelian extension $Z_{\lambda} / \mathbb{Q}$ can be expressed as a congruence condition on $p$. Conversely, by class field theory, if the splitting behaviour of $p$ in some field extension is determined by a congruence condition on $p$, then this extension must be abelian. In other words, the only information that is determined by a congruence condition on $p$ is the coset of $[G_{\lambda}, G_{\lambda}]$ inside $G_{\lambda}$ that $\text{Frob}_{\mathfrak{p}}$ belongs to. Therefore, if $f$ is weakly $\lambda$-abelian for some $x \in \mathbb{F}_{\lambda}$, then there must be a coset of $[G_{\lambda}, G_{\lambda}]$ on which this trace is constantly equal to $x$. The converse is now clear.
\end{proof}

\begin{opm}
In the remainder of this section, $k$ will always denote a finite field of characteristic $\ell$. In addition, we will let $K / k$ denote its unique quadratic extension and we denote $q = \# k$. The usage of these fields as opposed to $\mathbb{F}_{\lambda}$ and $\mathbb{F}_{\lambda^2}$ is to stress when certain statements are of more a general nature as opposed to tailored to the setting of the Galois representations associated with normalised eigenforms.
\end{opm}

\subsection{Classification of being totally $\lambda$-abelian}

The following lemma is well known.

\begin{lemma}\label{abgrp}
Let $G \subset \emph{GL}_2( k )$ be abelian. Then $G$ can be conjugated into the Borel subgroup of $\emph{GL}_2( K )$. 
\end{lemma}
\begin{proof}
Any eigenspace of a matrix that is not a multiple of the identity must be 1-dimensional. For any two such matrices $X,Y \in G$, if $v \in K^2$ is an eigenvector of $X$ with eigenvalue $\mu \in K^{\times}$, then $XYv = YXv = \mu Yv$ shows that $Yv$ is also in the $\mu$-eigenspace of $X$, and is thus a multiple of $v$. Therefore $v$ is also an eigenvector of $Y$. Now any basis that starts with $v$ will put $G$ inside the Borel subgroup.
\end{proof}

\begin{proof}[Proof of Theorem \ref{main1}.]
First suppose that the image $G_{\lambda}$ of $\rho_f^{\lambda}$ can be conjugated over $\mathbb{F}_{\lambda^2}$ to a subgroup of the Borel subgroup of $\text{GL}_2( \mathbb{F}_{\lambda^2} )$. By changing bases, we may regard $G_{\lambda}$ itself as such a subgroup. The map from $G_{\lambda}$ to $\mathbb{F}_{\lambda^2}^{\times} \times \mathbb{F}_{\lambda^2}^{\times}$ that sends a matrix to its two diagonal entries is now a group homomorphism. As the target is abelian, the image of $\text{Frob}_{\mathfrak{p}}$ under this map will be determined by a congruence condition on $p$. Because this image fully determines the trace, $f$ is totally $\lambda$-abelian in this case.

Conversely, if $f$ is totally $\lambda$-abelian, then Proposition \ref{red1} above implies in particular that the set of elements of trace $2 \in \mathbb{F}_{\lambda}$ can be written as the union of cosets of $[G_{\lambda}, G_{\lambda}]$ inside $G_{\lambda}$. Since $1 \in [G_{\lambda}, G_{\lambda}]$, this shows that each $A \in [G_{\lambda}, G_{\lambda}]$ must have trace 2. Since clearly $\text{det}(A) = 1$,  the characteristic polynomial of $A$ is given by $X^2 - 2X + 1 = (X-1)^2$. Therefore, $A$ must have the eigenvalue 1 with algebraic multiplicity 2, and so in the appropriate basis we may write
\[
A = \begin{pmatrix} 1 & x \\ 0 & 1 \end{pmatrix}, \quad \text{for some $x \in \mathbb{F}_{\lambda}$.}
\]
According to Proposition \ref{red1}, each set of constant trace should be stable under multiplication by $[G_{\lambda},G_{\lambda}]$, and so in particular, multiplication with $A$ should not change the trace. However,
\[
\text{tr}\left( \begin{pmatrix} a & b \\ c & d \end{pmatrix} \begin{pmatrix} 1 & x \\ 0 & 1 \end{pmatrix} \right) = \text{tr} \begin{pmatrix} a & b \\ c & d \end{pmatrix} + cx.
\]
Note that $x = 0$ if and only if $A = 1$. As soon as $[G_{\lambda}, G_{\lambda}] \neq \{ 1 \}$, it follows that in the right basis, $G_{\lambda}$ must be contained in the Borel subgroup of $\text{GL}_2( \mathbb{F}_{\lambda} )$. It remains to address the case that $[G_{\lambda}, G_{\lambda}] = \{ 1 \}$, or in other words, that $G_{\lambda}$ is abelian. The proof is then complete by Lemma \ref{abgrp} above.
\end{proof}

\subsection{Being weakly abelian for $x \in \mathbb{F}_{\lambda}^{\times}$}

We first prove the following generalisation of Lemma \ref{abgrp} above.
\begin{prop}\label{abgrp2}
Let $G \subset \emph{GL}_2( k )$ be a subgroup. Suppose that the identity is the only diagonalisable element of $[G, G]$. Then $G$ is conjugate to a subgroup of the Borel subgroup of $\emph{GL}_2( K )$. 
\end{prop}
\begin{proof}
If $[G, G] = \{ 1 \}$, then $G$ is abelian and by Lemma \ref{abgrp} above, it is conjugate to a subgroup of the Borel subgroup of $\text{GL}_2( K )$. We may therefore assume that $[G, G]$ contains a non-trivial element $A$. 

Every matrix in $[G, G] \setminus \{ 1 \}$ must have a repeated eigenvalue, and because its determinant is 1, this eigenvalue has to be either $1$ or $-1$. The trace of any element of $[G,G]$ must therefore be $2$ or $-2$. We claim that, in the appropriate basis, $[G, G]$ is a subgroup of the Borel subgroup. To see this, we choose a basis in which $A$ is upper-triangular with $\pm 1$ as both diagonal entries. Squaring this element if $\ell \neq 2$ shows that we can even find a matrix in $[G, G] \setminus \{ 1 \}$ with both diagonal entries equal to 1. Then
\[
A = \begin{pmatrix} 1 & \gamma \\ 0 & 1 \end{pmatrix} \! \! \quad\text{and} \quad  \! \!  \begin{pmatrix} a & b \\ c & d \end{pmatrix} \in [G, G] \implies \pm 2 = \text{tr} \left( \begin{pmatrix} 1 & \gamma \\ 0 & 1 \end{pmatrix}^n \begin{pmatrix} a & b \\ c & d \end{pmatrix} \right) = \text{ tr} \begin{pmatrix} a & b \\ c & d \end{pmatrix} + c n \gamma = \pm 2 + c n \gamma.
\]
In characteristic 2, since $\gamma \neq 0$, it is clear that $c = 0$. In odd characteristic $\ell$, if $c \neq 0$, the expression $c n \gamma$ will take on at least $\ell \geq 3$ values, and thus we can find $n$ such that the above equality is not satisfied. Therefore $c = 0$ in all cases and so $[G, G]$ is contained in the Borel subgroup.

We now show that $G$ itself must be contained in the Borel subgroup. Indeed, if one commutes a general matrix with an upper triangular matrix $A \in [G,G]$ as above, it is an easy check that the bottom left entry of the result is $-(ad-bc)^{-1}c^2 \gamma$. As $[G, G]$ is contained in the Borel subgroup and $\gamma \neq 0$, this should vanish and therefore $c = 0$; this concludes the proof.
\end{proof}

The following is the main result of this subsection.

\begin{stelling}\label{weaklab}
Let $f \in \mathcal{S}_{\bfk}( N, \chi )$ be a normalised eigenform and let $\lambda \subset \mathcal{O}_f$ be a prime. Suppose that $f$ is weakly $\lambda$-abelian for some $x \in \mathbb{F}_{\lambda}^{\times}$. Then $G_{\lambda}$ can be conjugated into the Borel subgroup inside $\emph{GL}_2( \mathbb{F}_{\lambda^2} )$.
\end{stelling}
\begin{proof}
Consider any $A \in [G_{\lambda}, G_{\lambda}]$ that is diagonalisable over $\mathbb{F}_{\lambda^2}$. Then since $\text{det}(A) = 1$, as it is a commutator, in the appropriate basis, we may write $A$ as the diagonal matrix with two entries $\alpha, \beta \in \mathbb{F}_{\lambda^2}$ with $\alpha \beta = 1$. By assumption there exists some coset of $[G_{\lambda}, G_{\lambda}]$ on which the trace is constant. Let $w$ and $z$ denote the two diagonal entries of any matrix representing this coset. 

Since multiplication by both $A$ and $A^{-1}$ should not change the trace, we find that $w \alpha + z \beta = w + z = z \alpha + w \beta$. Multiplying the first equation with $w$ and the second with $z$ yields that $w^2 \alpha + wz \beta = w^2 + wz$ and $z^2 \alpha + wz \beta = wz + z^2$. Subtracting these two equations now yields that $(w^2 - z^2) \alpha = w^2 - z^2$. 

Because $x = w + z \neq 0$, we may divide by $w+z$ to obtain $(w - z) \alpha = w - z$. If $w = z \neq 0$, then the system reduces to $\alpha + \beta = 2$, which together with $\alpha \beta = 1$ implies that $\alpha = \beta = 1$. Otherwise, we divide by $w-z \neq 0$ to also find that $\alpha = 1 = \beta$. The above shows that the identity is the only diagonalisable matrix in $[G_{\lambda}, G_{\lambda}]$ and therefore the proof is complete by Proposition \ref{abgrp2}.
\end{proof}

\subsection{Being weakly $\lambda$-abelian for $0 \in \mathbb{F}_{\lambda}$}

Recall that $\overline{G}_{\lambda}$ denotes the image of $G_{\lambda} \subset \text{GL}_2( \mathbb{F}_{\lambda} )$ under the projection map $\text{GL}_2( \mathbb{F}_{\lambda} ) \to \text{PGL}_2( \mathbb{F}_{\lambda} )$.

\begin{lemma}\label{abgrp3}
Let $G \subset \emph{GL}_2( k )$ be a subgroup and let $\overline{G} \subset \emph{PGL}_2( k )$ denote its projective image. If $\overline{G}$ is cyclic, then $G$ must be abelian.
\end{lemma}
\begin{proof}
The group $\overline{G}$ is a quotient of $G$ by a central subgroup. It is a famous fact, and in fact easy to show, that any group whose quotient by a central subgroup is cyclic, must in fact be abelian.
\end{proof}

Working in the algebraic group $\text{PGL}_2$ as opposed to $\text{GL}_2$ has a conceptual advantage. Namely, the traceless elements in $\text{PGL}_2( \mathbb{F}_{\lambda})$ can in fact be identified using a \emph{multiplicative} condition; this is not as easy in $\text{GL}_2( \mathbb{F}_{\lambda})$. This is captured by the following lemma.

\begin{lemma}\label{mat}
Let $k$ be a field and let $A \in \emph{PGL}_2(k)$. 
\begin{itemize}
\item If $\emph{char}(k) \neq 2$, then $A$ has order $2$ in $\emph{PGL}_2(k)$ if and only if $\emph{tr}(A) = 0$. 
\item If $\emph{char}(k) = 2$, then $A$ has order at most $2$ in $\emph{PGL}_2(k)$ if and only if $\emph{tr}(A) = 0$. 
\end{itemize}
\end{lemma}
\begin{proof}
Let $X$ be a lift of $A$ to $\text{GL}_2(k)$. The Cayley-Hamilton theorem states that $X^2 = \text{tr}(X)X - \text{det}(X)$.

Therefore, if $\text{tr}(X) = 0$, the matrix $X$ squares to a scalar matrix. Its image $A$ therefore has trivial square in $\text{PGL}_2(k)$, so its order is at most $2$. If its order were $1$, then $X$ itself were scalar, and therefore it could only be of zero trace if $\text{char}(k) = 2$.

Conversely, we note that $A$ has trivial square if and only if $X^2$ is a scalar matrix. This shows that $\text{tr}(X)X$ is also a scalar matrix in that case. If $\text{tr}(X) = 0$, we are done. If not, then it follows that $X$ itself must be scalar, but then $A$ was trivial to begin with.
\end{proof}

\begin{lemma}\label{dihedral}
Let $G \subset \emph{GL}_2( k )$ be a subgroup that consists only of diagonal or off-diagonal matrices. Then its projective image $\overline{G} \subset \emph{PGL}_2( k )$ is either cyclic or dihedral.
\end{lemma}
\begin{proof}
The subgroup of $G$ consisting of diagonal matrices is a subgroup of $k^{\times} \times k^{\times}$, and therefore its projective image is a subgroup of $k^{\times}$; whence it is cyclic. Projectively, all off-diagonal matrices square to the identity by Lemma \ref{mat} above. Therefore $\overline{G}$ consists of a cyclic subgroup of index at most 2 outside of which each element is of order exactly 2. This precisely means that $\overline{G}$ is either cyclic or dihedral.
\end{proof}

\begin{prop}\label{pm1}
Let $G \subset \emph{GL}_2( k )$ be a subgroup such that $[G,G] = \{ \pm 1 \}$. If $\ell \neq 2$, then $\overline{G} \cong C_2 \times C_2$. 
\end{prop}
\begin{proof}
By assumption, there exist $X,Y \in G$ such that $XY = -YX$. We may rewrite this as $Y^{-1}XY = - X$ and taking traces on both sides now yields that $\text{tr}( X ) = - \text{tr}( X )$. It follows that $\text{tr}(X) = 0$, so we may choose a basis in which $X$ equals the diagonal matrix with entries $x, -x \in K^{\times}$. One now computes that those matrices in $G$ that \emph{anti-commute} with $X$ must all be off-diagonal. Conversely, the matrices in $G$ that \emph{commute} with $X$ must all be diagonal. 

Since these two cases exhaust all elements of $G$, Lemma \ref{dihedral} above now shows that the group $\overline{G}$ must be cyclic or dihedral. If $\overline{G}$ were cyclic, then Lemma \ref{abgrp3} shows that $G$ itself must have been abelian, contradicting $-1 \in [G,G]$. Therefore $\overline{G}$ is dihedral. Since $[G,G] = \{ \pm 1 \}$, its image in $\overline{G}$ is trivial and therefore $\overline{G}$ must also be abelian. These two conditions force the conclusion of the lemma.
\end{proof}

\begin{lemma}\label{exclusive}
Let $G$ be a subgroup of $\emph{GL}_2( k )$ such that its projective image $\overline{G}$ is isomorphic to a dihedral group $D_n$ for some $n \geq 1$. Then $G$ can be conjugated into the Borel subgroup if and only if $n = 1$ or $\ell \mid n$.
\end{lemma}
\begin{proof}
This can be deduced from Theorem D in \cite{Faber} after noting that $D_{\ell}$ is an $\ell$-semi-elementary subgroup and therefore contained in the Borel subgroup, whereas the other dihedral cases are not. 

We can also argue directly: by conjugating it to be upper triangular, it is an easy check that the multiplicative order of the projective image of any diagonalisable matrix is coprime to $\ell$, whereas for non-diagonalisable matrices this order must equal $\ell$. Now suppose first that $G$ is contained in the Borel subgroup. Because $\overline{G}$ has an index 2 coset consisting of traceless elements, so must $G$. Let $H \subset G$ denote the relevant index 2 subgroup, so this coset is $G \setminus H$. The two diagonal entries of all elements of $G \setminus H$ must thus have opposite sign. Since $G \setminus H$ is stable under multiplication by any element of $H$, it follows that the two diagonal entries of any $A \in H$ must be equal. If $A$ is not diagonalisable, then its projective order is $\ell$ in $\overline{G}_{\lambda}$ and therefore $\ell \mid n$. If we cannot find such an element in $H$, then $\overline{H} = \{ 1 \}$ and so $n = 1$.

Conversely, if $n = 1$, then by Lemma \ref{abgrp3}, $G$ must be abelian and therefore a subgroup of the Borel subgroup by Lemma \ref{abgrp}. Suppose now that $\ell \mid n$, so that $\overline{G}$ contains an element of order $\ell$. It follows that $\overline{G}$ contains a non-diagonalisable element $\overline{A}$ as above and this element generates the cyclic subgroup of $\overline{G}$. Indeed, if the order of an element of $\overline{G}$ is divisible by $\ell$, it must equal $\ell$. The same argument as in the proof of Theorem \ref{main1} now shows that $\overline{G}$ must be contained in the Borel subgroup, and therefore so must $G$.
\end{proof}

\bigskip

\begin{stelling}\label{weaklab0}
Let $f \in \mathcal{S}_{\bfk}( N, \chi )$ be a normalised eigenform. Suppose that $f$ is weakly $\lambda$-abelian for $0 \in \mathbb{F}_{\lambda}$ and that $f$ is \emph{not} totally $\lambda$-abelian. Then $\overline{G}_{\lambda}$ isomorphic to a dihedral group $D_n$ with $n \neq 1$ and $\ell \nmid n$.
\end{stelling}

\begin{proof}
Since we assume $f$ to be not totally $\lambda$-abelian, the image $G_{\lambda}$ cannot be conjugated into a Borel subgroup. By Lemma \ref{abgrp2}, the commutator subgroup $[G_{\lambda}, G_{\lambda}]$ will therefore contain a diagonalisable matrix $A \neq 1$. If $[G_{\lambda}, G_{\lambda}] = \{ \pm 1 \}$, then Lemma \ref{pm1} above shows that $\overline{G}$ is dihedral, as $D_2 \cong C_2 \times C_2$. 

Suppose now that we can find $A \in [G,G] \setminus \{ \pm 1 \}$, and choose a basis in which $A$ is diagonal with entries $\alpha, \beta \in \mathbb{F}_{\lambda^2}^{\times}$. Since $\alpha \beta = 1$, we must have $\alpha \neq \beta$. The trace of the result of multiplying $A$ with any matrix with diagonal entries $x$ and $-x$ equals $( \alpha - \beta )x$. Since $G_{\lambda}$ should contain a coset of $[G_{\lambda}, G_{\lambda}]$ on which the trace is constantly zero by Lemma \ref{red1}, any element of this coset must now be off-diagonal. 

We next note that $[G_{\lambda}, G_{\lambda}]$ consists solely of diagonal matrices. Indeed, one easily sees that only diagonal matrices multiply an off-diagonal matrix into another off-diagonal matrix.

As a result, if we compute the commutator of an arbitrary matrix in $G_{\lambda}$ with $A$, then the off-diagonal entries of the result should vanish. Performing the calculation, we find the conditions $(\alpha^2 - 1) ab = 0 = (\beta^2 - 1) cd$. Since we assume $\alpha, \beta \neq \pm 1$, this shows that $ab = cd = 0$ for any matrix in $G_{\lambda}$. If $a \neq 0$, then $b = 0$ and we must have $d \neq 0$, and so $c = 0$ and the matrix is diagonal. If $b \neq 0$, then $a = 0$ and we must have $c \neq 0$, and so $d = 0$ and the matrix is off-diagonal. Now Lemma \ref{dihedral} shows that $\overline{G} \subset \text{PGL}_2( \mathbb{F}_{\lambda} )$ must be dihedral. Finally, since we assume $f$ to not be totally $\lambda$ abelian, Lemma \ref{exclusive} above in combination with Theorem \ref{main1} shows that $\overline{G} \cong D_n$ cannot satisfy $n = 1$ or $\ell \nmid n$. 
\end{proof}

\subsection{Classification of being (weakly) $\lambda$-abelian}\label{labclass}

Given Theorem \ref{weaklab}, to prove Theorems \ref{main2} and \ref{main3}, it remains to establish a converse of Theorem \ref{weaklab0} above. These are phrased in terms of the projective image $\overline{G}_{\lambda}$. Even though $\text{PGL}_2( \mathbb{F}_{\lambda})$ no longer comes equipped with a natural trace map, as the projection map $\text{GL}_2( \mathbb{F}_{\lambda}) \to \text{PGL}_2( \mathbb{F}_{\lambda})$ identifies many matrices with distinct traces, this projection map \emph{does preserve} the subset of traceless matrices. Therefore, we can still speak of traceless elements in $\text{PGL}_2( \mathbb{F}_{\lambda})$, and elements with non-zero trace.

\begin{lemma}\label{red2}
Let $f \in \mathcal{S}_{\bfk}( N, \chi )$ be a normalised eigenform. Then $f$ is $\lambda$-abelian for the class $0 \in \mathbb{F}_{\lambda}$ if and only if the set of traceless elements in $\overline{G}_{\lambda}$ can be written as a union of cosets of $[\overline{G}_{\lambda}, \overline{G}_{\lambda}] \subset \overline{G}_{\lambda}$.
\end{lemma}
\begin{proof}
As noted above, under the reduction map $\text{GL}_2( \mathbb{F}_{\lambda} ) \to \text{PGL}_2( \mathbb{F}_{\lambda} )$, both the set of traceless elements and its complement are preserved. Also, the image of a coset of $[G_{\lambda}, G_{\lambda}]$ in $G_{\lambda}$ is again a coset of $[\overline{G}_{\lambda}, \overline{G}_{\lambda}]$ in $\overline{G}_{\lambda}$, and the preimage of a coset is a union of cosets. The claim now follows from Proposition \ref{red1}.
\end{proof}

\begin{lemma}\label{com3}
Let $n$ be a positive integer and let $D_n = \langle r, s \mid r^n = s^2 = (sr)^2 = 1 \rangle$ be the dihedral group of cardinality $2n$. Then $[D_n, D_n] = \langle r^2 \rangle$. 
\end{lemma}
\begin{proof}
Every element of $D_n$ is either of the form $r^i$ or of the form $sr^{j}$ for $i,j \in \mathbb{Z}$. Two elements of the former shape commute, whereas one may compute that $[r^i, sr^j] = r^{2i}$ and $[sr^i, sr^j] = r^{2j-2i}$.
\end{proof}

\begin{proof}[Proof of Theorem \ref{main2}.]
Theorem \ref{weaklab} showed that if $f \in \mathcal{S}_{\bfk}( N, \chi )$ is weakly $\lambda$-abelian for some $x \in \mathbb{F}_{\lambda}$ but not totally $\lambda$-abelian, then $x = 0$. In addition, Theorem \ref{weaklab0} shows that if this happens, then $\overline{G}_{\lambda}$ must be isomorphic to a dihedral group $D_n$ with $n \neq 1$ and $\ell \nmid n$. It thus remains to show that if $\overline{G}_{\lambda}$ is isomorphic to such a group, then $f$ is weakly $\lambda$-abelian for $0 \in \mathbb{F}_{\lambda}$, but not totally $\lambda$-abelian. This second part is immediate from Lemma \ref{exclusive}. For the first, essentially by Lemma \ref{red2}, for this it suffices to show that there exists a coset of $[\overline{G}_{\lambda}, \overline{G}_{\lambda}]$ on which the trace is constantly zero. By Lemma \ref{com3} above, this commutator subgroup is contained in the cyclic subgroup of rotations inside $D_n$. Therefore, there exist cosets of $[\overline{G}_{\lambda}, \overline{G}_{\lambda}]$ that consist only of reflections. By Lemma \ref{mat}, these elements are all traceless and the proof is complete.
\end{proof}

\begin{proof}[Proof of Theorem \ref{main3}.]
Theorem \ref{main2} shows that $f$ is weakly $\lambda$-abelian for the class $0 \in \mathbb{F}_{\lambda}$, but not totally $\lambda$-abelian, if and only if the projective image $\overline{G}_{\lambda} \subset \text{PGL}_2( \mathbb{F}_{\lambda} )$ is isomorphic to a dihedral group $D_n$ with $n > 1$ and $\ell \nmid n$. It remains to investigate under which conditions the adjective ``weakly'' can be removed.

By Lemma \ref{red2}, we should check when the set of traceless matrices can be written as the union of cosets of $[ G_{\lambda}, G_{\lambda} ]$ in $G_{\lambda}$. We distinguish two cases. If $n$ is odd, then $[D_n, D_n]$ is of index 2 in $D_n$ by Lemma \ref{com3}, with one coset consisting of all $n$ rotations, none of which are involutions, and the other consisting of all $n$ reflections, all of which are involutions. It thus follows that $f$ must be $\lambda$-abelian if $\ell$ is odd.

If $n$ is even, then $\ell$ must be odd and $[D_n, D_n]$ is of index 4 in $D_n$ by Lemma \ref{com3}. This time, the involution $r^{n/2}$ is contained in the same coset as elements with a different order as soon as $n \geq 4$, and therefore $f$ cannot be $\lambda$-abelian. Only in case of $\overline{G}_{\lambda} \cong D_2 = C_2^2$ will $f$ be $\lambda$-abelian.
\end{proof}

\subsection{Dickson's subgroup classification theorem}\label{dicksonsec}

To investigate the condition of being semi-$\lambda$-abelian, we cannot rely on purely elementary considerations and we will need the following fundamental result by Dickson. It classifies what kinds of subgroups can occur in $\text{PGL}_2( k )$ for finite fields $k$, and was recently refined by Faber in \cite{Faber}.
\begin{stelling}\label{dickson}\emph{(Dickson)}
Let $\ell$ be a prime and let $G$ be a finite subgroup of $\emph{PGL}_2(\overline{\mathbb{F}}_{\ell})$. Then a conjugate of $G$ is one of the following groups:
\begin{itemize}
\item A finite subgroup of the Borel subgroup of upper triangular matrices;
\item $\emph{PSL}_2(\mathbb{F}_{\ell^r})$ or $\emph{PGL}_2(\mathbb{F}_{\ell^r})$ for some positive integer $r$;
\item A dihedral group $D_n$ for some positive integer $n$ coprime to $\ell$;
\item An exceptional subgroup isomorphic to either $A_4$, $S_4$ or $A_5$.
\end{itemize}
\end{stelling}
\begin{proof}
We appeal to the classification by Faber in Theorem D of \cite{Faber}, which characterises all conjugacy classes of subgroups in $\text{PGL}_2( \mathbb{F}_{\ell^r} )$ for any prime $\ell$ and integer $r \geq 1$, distinguishing 10 different cases. Cases 1, 2 and 10 can, possibly after passing to a quadratic field extension, be viewed as a subgroup of the upper triangular matrices. Cases 3, 4 and 5 are dihedral and cases 6, 7 and 8 are exceptional, leaving only case 9 concerning $\text{PSL}_2$ and $\text{PGL}_2$ coming from smaller fields.
\end{proof}

We stress that, especially when the finite fields in question are very small, these cases need not be mutually exclusive. For example, the Borel subgroup of $\text{PGL}_2( \mathbb{F}_4 ) \cong A_5$ is abstractly isomorphic to $A_4$. To illustrate the power of Dickson's result, we give an alternative proof of Theorem \ref{main3} as obtained in the previous subsections, which has the advantage of being comparatively quick, but the downside of being less transparent. Theorem \ref{dickson} inspires us to record the following sequence of lemmas.

\begin{lemma}\label{com2}
Let $k$ be a field with $\# k > 3$. Then $[\emph{SL}_2(k), \emph{SL}_2(k)] = \emph{SL}_2(k)$. As a result,
\[
[\emph{PSL}_2(k), \emph{PSL}_2(k)] = \emph{PSL}_2(k) \quad \text{and} \quad [ \emph{PGL}_2(k), \emph{PGL}_2(k) ] = \emph{PSL}_2(k).
\]
\end{lemma}
\begin{proof}
The first assertion is a standard fact; one first shows that elementary matrices generate the special linear group, and subsequently that every elementary matrix is a commutator if $\# k > 3$. We leave these checks to the curious reader. Since quotients of perfect groups are perfect, the claim about $\text{PSL}_2(k)$ follows immediately. Finally, since the determinant of any commutator must be 1, the final claim follows as well. 
\end{proof}

\begin{lemma}\label{com4}
We have that $[A_4, A_4] \cong C_2 \times C_2$. In addition, $[S_4,S_4] = A_4$ and $[A_5, A_5] = A_5$.
\end{lemma}
\begin{proof}
The Klein 4-group generated by the elements of order at most 2 is normal in $A_4$. In addition, it is well known that the commutator of $S_n$ is $A_n$ for $n \geq 4$, and that $A_n$ is simple for $n \geq 5$.  
\end{proof}

\begin{proof}[Second proof of Theorem \ref{main3}.]
By Theorem \ref{weaklab} and Lemma \ref{red2}, we should check whether the subset of traceless elements in $\overline{G}_{\lambda}$ is a union of cosets of $[\overline{G}_{\lambda}, \overline{G}_{\lambda}]$ in $\overline{G}_{\lambda}$. If a conjugate of $\overline{G}_{\lambda}$ is contained in the Borel subgroup of upper triangular matrices, Theorem \ref{main1} shows that $f$ is totally $\lambda$-abelian. We therefore reduce to checking the other cases from Theorem \ref{dickson}. Suppose that $f$ is $\lambda$-abelian for the class $0 \in \mathbb{F}_{\lambda}$.

First, if $\overline{G}_{\lambda}$ is equal to $\text{PSL}_2(\mathbb{F}_{\ell^r})$ or $\text{PGL}_2(\mathbb{F}_{\ell^r})$ for some positive integer $r$, then if $\ell^r > 3$, Lemma \ref{com2} shows that the set of traceless elements is not stable under multiplication by $[\overline{G}_{\lambda}, \overline{G}_{\lambda}]$, and therefore $f$ cannot be $\lambda$-abelian. If $\#k = 2$, then $\text{PSL}_2( \mathbb{F}_2 ) \cong \text{PGL}_2( \mathbb{F}_2 ) \cong S_3$ and if $\# k = 3$, then $\text{PSL}_2( \mathbb{F}_3 ) \cong A_4$ and $\text{PGL}_2( \mathbb{F}_3 ) \cong S_4$. Since these are dihedral or exceptional, we will address these cases momentarily. 

The case that $\overline{G}_{\lambda}$ is a dihedral group $D_n$ for some positive integer $n$ coprime to $\ell$ has already been explored in Subsection \ref{labclass}. Finally, if $\overline{G}_{\lambda}$ is isomorphic to either $A_4$, $S_4$ or $A_5$, we appeal to Lemma \ref{com4}. The commutator subgroup of $S_4$ is of index 2, but the trivial coset contains elements of both order 2 and order 3. By Lemma \ref{mat}, this means that it contains both traceless and non-traceless matrices; whence $f$ is not $\lambda$-abelian. In the case of $A_5$, the same conclusion follows immediately. In the case of $A_4$, the commutator subgroup consists of precisely those elements of order at most 2, and therefore $f$ only is $\lambda$-abelian in case $\ell = 2$. By Theorem D.6 in \cite{Faber}, if $A_4 \subset \text{GL}_2( \mathbb{F}_{\lambda} )$, then $[ \mathbb{F}_{\lambda} : \mathbb{F}_2 ]$ must be even. However, $A_4 \subset \text{GL}_2( \mathbb{F}_4 )$ is precisely the full Borel subgroup of upper triangular matrices, so now $f$ is in fact totally $\lambda$-abelian.

We have now shown that if $f$ is $\lambda$-abelian, then either $f$ is totally $\lambda$-abelian, or $\overline{G}_{\lambda} \cong D_n$ for $n = 2$ or some odd integer $n \geq 1$ such that $\ell \nmid n$. It remains to verify that for $n \neq 1$, the eigenform $f$ cannot be totally $\lambda$-abelian. This is immediate by Lemma \ref{exclusive} and Theorem \ref{main1} and the proof is complete.
\end{proof}

\subsection{Classification of being semi-$\lambda$-abelian}

The following result transforms the condition of being semi-$\lambda$-abelian into group theory.
\begin{lemma}\label{red3}
A normalised cuspidal eigenform $f \in \mathcal{S}_{\bfk}( N, \chi )$ is semi-$\lambda$-abelian for some $f$-proper class $x \in \mathbb{F}_{\lambda}$ if and only if there is a coset of $[G_{\lambda}, G_{\lambda}]$ in which no element has trace $x$. 
\end{lemma}
\begin{proof}
Recall from the proof of Proposition \ref{red1} that the only information that is determined by a congruence condition on $p$ is the coset of $[G_{\lambda}, G_{\lambda}]$ inside $G_{\lambda}$ that $\text{Frob}_{\mathfrak{p}}$ belongs to. It is now clear that we can write down a \emph{proper} subset of congruence classes if and only if one of the cosets of $[G_{\lambda}, G_{\lambda}]$ misses the trace $x$.
\end{proof}

\begin{lemma}\label{cominj}
Let $G \subset \emph{GL}_2( \mathbb{F}_{\lambda} )$ be a subgroup. Then the map
\[
[G,G] \to [\overline{G}, \overline{G}] \subset \emph{PSL}_2( \mathbb{F}_{\lambda} ) 
\]
is either an isomorphism, or has a kernel of size 2, in which case $\ell \neq 2$.
\end{lemma}
\begin{proof}
It is clear that the map is surjective. The claim about the kernel is immediate from the observation that this kernel consists of all scalar matrices in $[G,G]$. Since all of these matrices have determinant 1, this leaves only room for the identity and its negative.
\end{proof}

We will use Dickson's theorem to reduce to a case distinction as illustrated in the previous subsection. By Theorem \ref{main1}, we may disregard the case that $G_{\lambda}$ can be conjugated into the Borel subgroup.

\begin{lemma}\label{trlem1}
Let $G \subset \emph{GL}_2( k )$ be a subgroup and let $F \subset k$ be a subfield with $\# F > 3$. Suppose that $\overline{G} \subset \emph{PGL}_2( k )$ is isomorphic to $\emph{PSL}_2( F )$ or $\emph{PGL}_2( F )$. The image of the trace map on $[G,G]$ then equals $F$.
\end{lemma}
\begin{proof}
By Theorem D in \cite{Faber}, there are unique conjugacy classes of subgroups isomorphic to $\text{PSL}_2( F )$ and $\text{PGL}_2( F )$. It follows that in the right basis we may assume that $\overline{G}$ is in fact equal to the canonical subgroups $\text{PSL}_2( F )$ and $\text{PGL}_2( F )$ respectively. From Lemma \ref{com2} it now follows that $[ \overline{G}, \overline{G} ] \cong \text{PSL}_2( F )$. The intersection of the preimage of this group under the projection map $\text{GL}_2( k ) \to \text{PGL}_2( k )$ with the subgroup $\text{SL}_2( k )$ is precisely $\text{SL}_2( F )$. It follows from Lemma \ref{cominj} that $[G,G] \subset \text{SL}_2( F )$ is a subgroup of index at most 2. But again by Lemma \ref{com2}, it holds that $[\text{SL}_2( F ), \text{SL}_2( F )] = \text{SL}_2( F )$ and therefore $\text{SL}_2( F )$ does not contain an index 2 subgroup. We conclude that $[G,G] = \text{SL}_2( F )$ and we are done.
\end{proof}

\begin{lemma}\label{trlem2}
Let $G \subset \emph{GL}_2( k )$ be a subgroup with the property that $\overline{G} \cong A_4 \subset \emph{PGL}_2( k )$. Suppose that $\emph{char}(k) = \ell$ is odd. Then the image of the trace map on $[G,G]$ equals $\{ 0, \pm 2 \}$.
\end{lemma}
\begin{proof}
From Lemma \ref{com4} we deduce that $[ \overline{G}, \overline{G} ] \cong C_2 \times C_2$ consists of three involutions and the trivial class. Since the former will be traceless by Lemma \ref{mat} and the latter can only lift to $\pm1$, to prove the lemma we need only show that $-1 \in [G,G]$. Proposition 4.13 in \cite{Faber} shows that in the right $K$-basis, the preimage of this group in $\text{SL}_2( K )$ is contained in the subgroup consisting of the matrices
\[
\pm \begin{pmatrix} 1 & 0 \\ 0 & 1 \end{pmatrix}, \quad \pm \begin{pmatrix} i & 0 \\ 0 & -i \end{pmatrix}, \quad \pm \begin{pmatrix} 0 & 1 \\ -1 & 0 \end{pmatrix} \quad \text{and} \quad \pm \begin{pmatrix} 0 & i \\ i & 0 \end{pmatrix}.
\]
Since the second matrix squares to $-1$, it follows that $-1 \in [G,G]$ and we are done.
\end{proof}

\begin{lemma}\label{trlem3}
Let $G \subset \emph{GL}_2( k )$ be a subgroup with the property that  $\overline{G} \cong S_4 \subset \emph{PGL}_2( k )$. Then $\ell \neq 2$ and the image of the trace map on $[G,G]$ equals $\{ 0, \pm 1, \pm 2 \}$.
\end{lemma}
\begin{proof}
From Lemma \ref{com4} we deduce that $[ \overline{G}, \overline{G} ] \cong A_4$. By Proposition 4.13 in \cite{Faber} we have $\ell \neq 2$, and in combination with the observation that $-1 \in [G,G]$ from the proof of Lemma \ref{trlem2} above, it follows that in the right $K$-basis, the group $[G,G]$ is equal to the subgroup generated by the matrices from the proof of Lemma \ref{trlem2} above and the matrix
\[
\frac{1}{2}\begin{pmatrix} 1+i & -1-i \\ 1+i & 1-i \end{pmatrix}.
\]
The claim from the lemma now follows from a routine computation.
\end{proof}

\begin{lemma}\label{trlem4}
Let $G \subset \emph{GL}_2( k )$ be a subgroup with the property that $\overline{G} \cong A_5 \subset \emph{PGL}_2( k )$. Then if $\varphi = ( 1 + \sqrt{5} ) / 2$ denotes the golden ratio, the image of the trace map on $[G,G]$ equals 
\[
\{ 0, \pm 1, \pm 2, \pm \varphi, \pm ( \varphi - 1) \}.
\]
\end{lemma}
\begin{proof}
By Theorem D.8 in \cite{Faber}, this case can only happen if $\# k \equiv 0, \pm 1 \bmod 5$, which ensures that $\varphi \in k$ if we write $\varphi = 3 \in \mathbb{F}_5$. From Lemma \ref{com4} we deduce that $[ \overline{G}, \overline{G} ] \cong A_5$. Proposition 4.21 in \cite{Faber} gives explicit generators for this subgroup in the right $K$-basis. In combination with the observation that $-1 \in [G,G]$ from before, it follows that generators for $[G, G]$ can be given by
\[
\pm \begin{pmatrix} \zeta_{10} & 0 \\ 0 & \zeta_{10}^{-1} \end{pmatrix} \quad \text{and} \quad \pm ( \zeta_{10} - \zeta_{10}^{-1} ) \cdot \begin{pmatrix} 1 & 2 - \zeta_{10} - \zeta_{10}^{-1} \\ 1 & -1 \end{pmatrix}.
\]
Again, the rest is now a computation. 
\end{proof}
\begin{opm}
An alternative, but less transparent, way to prove the Lemmas \ref{trlem2}, \ref{trlem3} and \ref{trlem4} above is to show by abstract means that $[G,G] \cong Q_8$, the quaternion group, $[G,G] \cong \text{SL}_2( \mathbb{F}_3 )$ and $[G,G] \cong \text{SL}_2( \mathbb{F}_5 )$ respectively. The traces of the irreducible faithful symplectic representations of these groups are known.
\end{opm}

\begin{proof}[Proof of Theorem \ref{main4}.]
Suppose first that $\overline{G}_{\lambda} \cong \text{PSL}_2( k )$ or $\text{PGL}_2( k )$ for some proper subfield $k \subsetneq \mathbb{F}_{\lambda}$ with $\# k > 3$. Then Lemma \ref{trlem1} shows that the traces that occur on $[G_{\lambda}, G_{\lambda}]$ are precisely the values in $k$, as well as the determinants for $k^{\times}$. If the determinant is not $k$-valued, then there exists some scalar $\alpha \cdot \text{id} \in G_{\lambda}$ with $\alpha \notin k$, and therefore there is a coset of $[G_{\lambda}, G_{\lambda}]$ for which the traces are given by $\alpha \cdot k$. Its intersection with $k$ is $\{ 0 \}$, so it follows that $f$ is semi-$\lambda$-abelian for all $f$-proper classes $x \in \mathbb{F}_{\lambda}^{\times}$. On the other hand, the class $0 \in \mathbb{F}_{\lambda}$ can never be semi-$\lambda$-abelian, because both $\text{PSL}_2( k )$ and its complement in $\text{PGL}_2( k )$ always contain traceless involutions, and therefore so will all their lifted cosets of $[G_{\lambda}, G_{\lambda}]$ in $G_{\lambda}$. Finally, if $\overline{G}_{\lambda} \cong \text{PSL}_2( \mathbb{F}_{\lambda} )$ or $\text{PGL}_2( \mathbb{F}_{\lambda} )$, then the trace is surjective on all cosets and therefore $f$ is not semi-$\lambda$-abelian for any class in $\mathbb{F}_{\lambda}$. As before, we will address the cases in which $\# k \leq 3$ momentarily.

Suppose now that $\overline{G}_{\lambda}$ is isomorphic to a dihedral group $D_n$ for some $n > 1$ coprime to $\ell$. We claim that $f$ is semi-$\lambda$-abelian for all $f$-proper classes $x \in \mathbb{F}_{\lambda}$ except for $x = 0$ if $\ell = 2$. Indeed, if $x \in \mathbb{F}_{\lambda}^{\times}$, then by Lemma \ref{com3}, there is a coset of $[G_{\lambda}, G_{\lambda}]$ that projectively consists of only traceless reflections, so in particular the value $x$ is not assumed. If $x = 0$, it suffices to find any coset of $[G_{\lambda}, G_{\lambda}]$ that does \emph{not} contain a traceless element. Suppose first that $\ell \neq 2$. If $n > 1$ is odd, then $[G_{\lambda}, G_{\lambda}]$ itself is such a coset, since its projective image is the unique cyclic subgroup of odd order $n$ of $D_n$ by Lemma \ref{com3}. It does not contain an involution, and so no traceless elements by Lemma \ref{mat}. If $n \equiv 2 \bmod 4$, the same argument as in the case that $n$ is odd will yield the right conclusion. If $n \equiv 0 \bmod 4$, then we consider any coset of $[G_{\lambda}, G_{\lambda}]$ that reduces to the non-trivial coset of $[D_n, D_n] = \langle r^2 \rangle$ represented by $r$. This will not contain any involutions, proving our claim. Finally, if $\ell = 2$, we need only consider the case that $n$ is odd. But since the identity is traceless, now both cosets will contain traceless elements, so $f$ is not semi-$\lambda$-abelian for $0$ in this case.

It remains to analyse the three exceptional images. Suppose first that $\overline{G}_{\lambda} \cong A_4$, in which case we may assume that $\ell \neq 2$ because $A_4 \subset \text{GL}_2( \mathbb{F}_4 )$ describes the Borel subgroup. By Lemma \ref{trlem2}, the subgroup $[G_{\lambda}, G_{\lambda}]$ contains precisely the traces $\{ 0, \pm 2 \}$. It follows that for any $f$-proper class $x \notin \{ 0, \pm 2 \}$, the eigenform $f$ must be semi-$\lambda$-abelian. In addition, it is also always semi-$\lambda$-abelian for the class $0 \in \mathbb{F}_{\lambda}$, because any coset of $[G_{\lambda}, G_{\lambda}]$ lifting a non-trivial coset of $[\overline{G}_{\lambda}, \overline{G}_{\lambda}]$ does not contain any traceless elements, since these elements are projectively all of order 3. Finally, by Lemma \ref{trlem3}, the cosets of $[\overline{G}_{\lambda}, \overline{G}_{\lambda}]$ lifting a non-trivial coset of $[A_4, A_4]$ in $A_4$ will have traces $\{ \pm \alpha \}$ for some $\alpha \in \mathbb{F}_{\lambda}^{\times}$. This shows that $f$ is also $\lambda$-abelian for the classes $\pm 2$ as soon as $\alpha \neq \pm 2$. But then squaring the matrix will still yield the desired conclusion unless $\ell = 3$, for which we would need a scalar $\alpha \cdot \text{id} \in G_{\lambda}$ with $\alpha \in \mathbb{F}_{\lambda}^{\times} \setminus \mathbb{F}_3^{\times}$. 

Now suppose that $\overline{G}_{\lambda} \cong S_4$, so that $[ \overline{G}_{\lambda}, \overline{G}_{\lambda} ] \cong A_4$. By Theorem D.7 in \cite{Faber}, this can only happen if $\ell \neq 2$. By Lemma \ref{trlem3}, the subgroup $[G_{\lambda}, G_{\lambda}]$ contains precisely the traces $\{ 0, \pm 1, \pm 2 \}$. It follows that for any $f$-proper class $x \notin \{ 0, \pm 1, \pm 2 \}$, the eigenform $f$ must be semi-$\lambda$-abelian. We next note that $f$ is never semi-$\lambda$-abelian for the class $0 \in \mathbb{F}_{\lambda}$, since both $A_4$ and $S_4 \setminus A_4$ contain traceless involutions. From Proposition 4.16 in \cite{Faber} it follows that any matrix representing the non-trivial coset is projectively given by a diagonal matrix with the entries $i$ and $1$. It is now quickly checked that the only non-zero traces on such a coset of $[G_{\lambda}, G_{\lambda}]$ are equal to $\pm \alpha (1+i)$ for some $\alpha \in \mathbb{F}_{\lambda^2}^{\times}$. For this set to equal $\{ \pm 1 \}$, we must have $\pm \alpha = (1-i) / 2$. Cubing this matrix also yields a coset with its only non-zero traces $\pm \alpha^3( 1 - i ) = \pm 1/2$. Now $\{ \pm 1 \} = \{ \pm 1/2 \}$ if and only if $\ell = 3$. Similarly, one finds that $\{ \pm \alpha (1+i) \}$ equals $\{ \pm 2 \}$ if and only if $\pm \alpha = 1-i$, and then $\pm \alpha^3( 1 - i ) = \pm 4$. Again, this coincides with $\pm 2$ if and only if $\ell = 3$, since $\ell \neq 2$. In the case that $\ell = 3$ and $\alpha$ is as above, then the determinant of the matrix representing the non-trivial coset equals $i \cdot (1-i)^2 = 2 = -1$, and therefore the determinant takes values in $\{ \pm 1 \}$.

Finally, suppose that $\overline{G}_{\lambda} \cong A_5$, so that also $[ \overline{G}_{\lambda}, \overline{G}_{\lambda} ] \cong A_5$. By Lemma \ref{trlem4}, the subgroup $[G_{\lambda}, G_{\lambda}]$ contains precisely the traces $\{ 0, \pm 1, \pm 2, \pm \varphi, \pm (\varphi -1 ) \}$. It follows that for other any $f$-proper class, the eigenform $f$ must be semi-$\lambda$-abelian for $x$. Since $[G_{\lambda}, G_{\lambda}]$ surjects onto $\overline{G}_{\lambda}$, all other cosets of $[G_{\lambda}, G_{\lambda}]$ are simply scaled versions of it. In particular, all cosets will contain traceless elements, and $f$ is never semi-$\lambda$-abelian for $0 \in \mathbb{F}_{\lambda}$. For the classes in $T \colonequals \{ \pm 1, \pm 2, \pm \varphi, \pm (\varphi -1 ) \}$, we should investigate for which finite fields $\mathbb{F}_{\lambda}$ there exists some $\alpha \in \mathbb{F}_{\lambda}^{\times} \setminus \{ \pm 1 \}$ such that the intersection of all subsets $\alpha^n T$ for $n \in \mathbb{Z}$ is non-empty. Since $\rho_f^{\lambda}$ is odd, we always have $\alpha \cdot \text{id} \in G_{\lambda}$ with $\alpha$ of multiplicative order 4, since only then $\text{det}( \alpha \cdot \text{id} ) = -1$ if $\ell \neq 2$. For such $\alpha$, it is easy to see that $T \cap \alpha T \neq \varnothing$ for only finitely many $\ell$. This reduces the problem to a finite computer search, the results of which are summarised in the statement of Theorem \ref{main4}.

We will address the values of $c_{\lambda}$ in the next subsection.
\end{proof}

We remark here that the proof above yields an alternative, but less transparent, proof of Theorem \ref{weaklab}.

\subsection{Analysing the congruence conditions}

To prove Proposition \ref{mod}, we will need a bound for the conductor of an abelian extension of $\mathbb{Q}$ with fixed degree and controlled ramification. This is established in the following two results.
\begin{prop}\label{cond1}
Let $S$ be a finite set of rational primes and let $q = p^r$ be a prime power. Let $L / \mathbb{Q}$ be a Galois extension with $\emph{Gal}(L / \mathbb{Q}) \cong C_q$ that is unramified outside $S$. Let $s$ denote the product of all primes in $S$. Then $L \subset \mathbb{Q}( \zeta_n )$ where $n = s$ if $p \notin S$ and $n = 2qs$ otherwise.
\end{prop}
\begin{proof}
The local proof of the Kronecker-Weber theorem shows that, if we can find exponents $n_{\ell} = \ell^{k_{\ell}} m_{\ell} \in \mathbb{N}$ with $\ell \nmid m_{\ell}$ such that the completion $L_{\mathfrak{l}}$ at any prime $\mathfrak{l}$ of $L$ above $\ell$ is contained in the field $\mathbb{Q}_{\ell}( \zeta_{n_{\ell}} )$, then globally, $L \subset \mathbb{Q}( \zeta_n )$ for $n = \prod_{\ell \in S} \ell^{k_{\ell}}$. If $\ell \neq p$, then $L$ is tamely ramified at $\ell$ and therefore, locally one may choose $k_\ell = 1$. If $\ell = p$, then one may always choose $k_p = r+1$ if $p$ is odd, and $k_p = r + 2$ if $p = 2$, see Chapter 6 in Stevenhagen's lecture notes \cite{stevenhagen}.
\end{proof}

Recall that for a finite set of rational primes $S$ and a positive integer $n$, we defined $\text{gcd}(n,S)$ as the part of $n$ supported at the primes occurring in $S$.

\begin{gevolg}\label{cond2}
Let $S$ be a finite set of rational primes and let $L / \mathbb{Q}$ be a finite abelian extension that is unramified outside $S$. Let $s$ denote the product of all primes in $S$. Then $L \subset \mathbb{Q}( \zeta_n )$ where $$n = s \cdot \emph{gcd}(2 \emph{ exp}( \emph{Gal}( L / \mathbb{Q} ) ), S).$$
\end{gevolg}
\begin{proof}
A finite abelian group can be written as the product of subgroups of prime power order, and therefore $L$ can be seen as the compositum of subfields of prime power order. Then $L$ is contained in $\mathbb{Q}( \zeta_n )$ for $n$ the least common multiple of the numbers from Proposition \ref{cond1}. If multiple cyclic subgroups share the same base prime, only the largest will contribute; hence the appearance of the exponent.
\end{proof}

\begin{proof}[Proof of Proposition \ref{mod}]
In the Borel case, we apply Corollary \ref{cond2} to Theorem \ref{main1}, with $S$ the set of rational primes dividing $N \ell$, the product of which is equal to $\text{rad}( N\ell)$, and the field $L$ whose Galois group over $\mathbb{Q}$ can be identified with $\text{im}( \rho_f^{\text{ss}, \lambda} )$ over the rationals. Since the image of Frobenius in $L$ determines the trace, the result now follows directly from Corollary \ref{cond2} above. 

In the dihedral cases $D_n$ with $n \geq 3$, the proofs of Theorems \ref{main2} and \ref{main3} show that the coset(s) of traceless matrices are given by the complement of the cyclic subgroup of rotations inside $D_n$ of index 2. The relevant extension corresponding to this cyclic subgroup is then quadratic. The bound on the divisor $M$ now follows from Corollary \ref{cond2} after observing that $\text{gcd}( 4, S ) = \text{gcd}(2, N\ell )^2$. It is well known that the splitting behaviour of primes in quadratic extensions is governed by the Legendre symbol for the discriminant. If $N \ell$ is odd, then $2$ will be unramified, forcing the discriminant of the quadratic field to be $1 \bmod 4$.

Similarly, in the case of $D_2 \cong C_2 \times C_2$, the relevant field will in fact be biquadratic and for $\ell \neq 2$, all three non-trivial classes will be traceless. Frobenius at $p$ belongs to these as soon as $p$ is not completely split.
\end{proof}

\begin{proof}[Proof of Proposition \ref{goodmod}]
Our assumption that $f \in \mathcal{S}_{\bfk}( \Gamma_0(N) )$ means that $\chi = \mathbbm{1}$ is the trivial character, and therefore by Deligne's theorem, $\text{det}( \rho_{f, \lambda}( \text{Frob}_{\mathfrak{p}} ) ) = p^{\bfk-1}$ for all primes $p \nmid N \ell$. By Chebotarev's density theorem, every (conjugacy) class in $G_{\lambda}$ is represented by $\text{Frob}_{\mathfrak{p}}$ for some prime $\mathfrak{p} \subset \overline{\mathbb{Z}}$. The group $G_{\lambda}$ in $\text{GL}_2( \mathbb{F}_{\lambda} )$ therefore comes with a natural $\mathbb{F}_{\ell}$\emph{-valued} determinant map $\text{det} : G_{\lambda} \to \mathbb{F}_{\ell}^{\times}$. 

If $\alpha \cdot \text{id} \in G_{\lambda} \subset \text{GL}_2( \mathbb{F}_{\lambda} )$ is a scalar matrix for some $\alpha \in \mathbb{F}_{\lambda}^{\times}$, then its determinant $\alpha^2$ must be a member of $\mathbb{F}_{\ell}^{\times}$. Therefore $[ \mathbb{F}_{\ell}( \alpha ) : \mathbb{F}_{\ell} ]$ divides both 2 and the odd number $[ \mathbb{F}_{\lambda} : \mathbb{F}_{\ell} ]$; whence it must be 1 and we find that $\alpha \in \mathbb{F}_{\ell}^{\times}$. In other words, the determinant of any scalar matrix in $G_{\lambda}$ must be a square in $\mathbb{F}_{\ell}^{\times}$. It follows that the determinant map descends to a map $\overline{\text{det}} : \overline{G}_{\lambda} \to \mathbb{F}_{\ell}^{\times} / ( \mathbb{F}_{\ell}^{\times} )^2$.  

If $p$ is not a square modulo $\ell$, then as $\bfk$ is assumed even, neither is $p^{\bfk-1}$. Therefore the image $\text{det}(G_{\lambda})$ contains both square and non-square values in $\mathbb{F}_{\ell}$. The map $\overline{\text{det}} : \overline{G}_{\lambda} \to \mathbb{F}_{\ell}^{\times} / ( \mathbb{F}_{\ell}^{\times} )^2$ is therefore surjective, and so there is an index 2 subgroup of $\overline{G}_{\lambda}$ given by its kernel, to which $\text{Frob}_{\mathfrak{p}}$ belongs if and only if its determinant $p^{\bfk-1}$, or equivalently, $p$ itself, is a square modulo $\ell$.

On the other hand, if $\overline{G}_{\lambda} \cong D_n$ for some odd $n$, then from the fact that $[D_n : [D_n, D_n]] = 2$ by Lemma \ref{com3}, one immediately deduces that $[\overline{G}_{\lambda}, \overline{G}_{\lambda}]$ is the only index 2 subgroup of $\overline{G}_{\lambda}$, and therefore it must coincide with the index 2 subgroup found above. We have seen that the traces of precisely those elements in the non-trivial coset of this commutator subgroup vanish mod $\lambda$. This therefore happens if and only if $p$ is not a square modulo $\ell$, completing the proof of the proposition.
\end{proof}

Recall that for a normalised cuspidal eigenform $f \in \mathcal{S}_{\bfk}( N, \chi )$, we defined $c_{\lambda} \in [0,1]$ as the density of primes $p$ for which $a_p(f) \equiv 0 \bmod \lambda$. We note that this density exists as a consequence of Chebotarev's density theorem, since this condition is governed by behaviour of Frobenius in the number field $L_{\lambda}$ defined at the start of this section. To complete the proof of Theorem \ref{main4}, it remains to compute the values $c_{\lambda}$ in all the cases of Dickson's Theorem \ref{dickson}. We will need the following lemma; recall that $\#k = q$.
\begin{lemma}\label{tr0count}
The number of traceless elements in $\emph{PGL}_2( k )$ is equal to $q^2$. If $\ell \neq 2$, the number of traceless elements in $\emph{PSL}_2( k )$ is equal to $q(q+1)/2$ if $q \equiv 1 \bmod 4$ and $q(q-1)/2$ if $q \equiv -1 \bmod 4$.
\end{lemma}
\begin{proof}
There are two types of traceless elements in $\text{PGL}_2( k )$, of which there are $q-1$ elements with two zeroes on the diagonal. All other elements can be uniquely normalised to have the diagonal elements $1$ and $-1$. These matrices are in bijection with $x, y \in k$ with $-(1+xy) \neq 0$. Conditioning on $x$ being zero or not, we find $q + (q-1)^2 = q^2 - q + 1$ such elements, proving the first claim. A similar argument in $\text{PSL}_2( k )$ yields $(q-1)/2$ elements of the first kind, and of the second kind no elements with $x = 0$ if $-1$ is not a square, and $q$ otherwise. Combining these observations completes the proof.
\end{proof}

\begin{proof}[Completing the proof of Theorem \ref{main4}.]
We once again appeal to Dickson's Theorem \ref{dickson} to reduce to a case distinction, recalling that the map $\text{GL}_2( \mathbb{F}_{\lambda} ) \to \text{PGL}_2( \mathbb{F}_{\lambda} )$ preserves the subset of traceless matrices.

If $\overline{G}_{\lambda}$ can be conjugated into the Borel subgroup in $\text{PGL}_2( \mathbb{F}_{\lambda^2} )$, then the diagonal entries are given by two characters $\psi_1, \psi_2 : G_{\lambda} \to \mathbb{F}_{\lambda^2}^{\times}$. The subset of traceless elements of this group are then those elements in the preimage of $-1$ of the character $\psi_1 / \psi_2$. If $-1$ does not appear in its image, then $c_{\lambda} = 0$. If it does, then the size of its preimage is equal to the multiplicative order of the character. If the characters are $\mathbb{F}_{\lambda}^{\times}$-valued, then this order must divide $q-1$, where $q = \# \mathbb{F}_{\lambda}$. 

If not, then because their sum and product are $\mathbb{F}_{\lambda}$-valued, the charaters must be Galois conjugates. Recall that $\mathbb{F}_{\lambda^2}^{\times}$ is a cyclic group of order $q^2 - 1$, within which the group $\mathbb{F}_{\lambda}^{\times}$ is contained as the unique cyclic subgroup of order $q - 1$. Therefore, raising to the power $q + 1$ maps $\mathbb{F}_{\lambda^2}$ into $\mathbb{F}_{\lambda}$. It follows that the order of $\psi_1 / \psi_2$ must divide $q+1$. If $\ell \neq 2$, then because this character assumes the value $-1$, its order must be even.

If the image of $\overline{\rho}_f^{\lambda}$ is isomorphic to $\text{PGL}_2( k )$ for some subfield $k \subset \mathbb{F}_{\lambda}$ with $\# k = q$, then recalling the well-known fact that $\# \text{PGL}_2( k ) = q(q-1)(q+1)$, the claim is immediate from Chebotarev's density theorem in combination with Lemma \ref{tr0count} above. The proof in case of $\text{PSL}_2( k )$ is very similar.

For the remaining cases, we appeal to Lemma \ref{mat} to reduce to counting the number of involutions. If the image of $\overline{\rho}_f^{\lambda}$ is isomorphic to a dihedral group $D_n$ with $n$ odd, then all involutions are contained in the non-trivial coset of rotations inside $D_n$. If $n$ is even, then this subgroup contains one additional involution. The result now follows by recalling that for $\ell = 2$, also the trivial element is traceless.

For the three exceptional cases, the proof follows from counting the number of involutions in $A_4$, $S_4$ and $A_5$ respectively. We leave those elementary details to the reader.
\end{proof}

\section{Examples}\label{exam}

We illustrate our results in the setting of Galois representations attached to classical normalised cuspidal eigenforms of weight 2 with trivial nebentypus, because these representations are most accessible due to their ostensible connections to the theory of elliptic curves and the LMFDB \cite{LMFDB}.

\begin{opm}\label{CM}
If an elliptic curve $E / \mathbb{Q}$ has complex multiplication by the maximal order inside an imaginary quadratic number field of discriminant $D < 0$, then $E$ is supersingular at some prime $p$ if and only if $p$ is inert in the extension $\mathbb{Q}( \sqrt{D} ) / \mathbb{Q}$. If $p > 3$, this is equivalent to the condition that $a_p(E) = 0$. 

It is well known, see for example Proposition 1.14 in \cite{zywina}, that if $j(E) \neq 0$, the mod $\ell$ image of $\rho_E^{\ell}$ for odd $\ell \nmid D$ is equal to the normaliser of a split Cartan or a non-split Cartan subgroup, again depending on the splitting behaviour of $\ell$ in $\mathbb{Q}(\sqrt{D}) / \mathbb{Q}$. These Cartan subgroups are abstractly isomorphic to $\mathbb{F}_{\ell}^{\times} \times \mathbb{F}_{\ell}^{\times}$ and $\mathbb{F}_{\ell^2}^{\times}$ respectively, and therefore their projective images are cyclic of order $\ell - 1$ and $\ell + 1$ respectively. Since the Cartan subgroups are of index 2 in their normalisers, the image of $\overline{\rho}_E^{\ell}$ is isomorphic to $D_{\ell - 1}$ or $D_{\ell + 1}$ respectively. Therefore, unless $\ell = 3$ in the split Cartan case, the newform associated with $E$ through modularity is not $\ell$-abelian by Theorem \ref{main3}.

If $\ell \mid D$, then Proposition 1.14 in \cite{zywina} shows that the mod $\ell$ image of $\rho_E^{\ell}$ is contained in the Borel subgroup, and so in this case, the relevant newform is totally $\ell$-abelian.
\end{opm}

Remark \ref{CM} explains that the mod $\ell$ images of rational elliptic curves with complex multiplication are well understood, but not a rich source of $\ell$-abelian modular forms that are not also totally $\ell$-abelian. The work of Zywina \cite{zywina} classifies conjecturally, but often with proof, all the possible rational non-CM elliptic curves whose Galois representations modulo certain primes are not surjective. From these results the following can be gleaned for non-CM elliptic curves $E / \mathbb{Q}$:

\begin{itemize}
\item If the image of $\rho_E^{\ell}$ is contained in the Borel subgroup, then $\ell \in \{ 2, 3, 5, 7, 11, 13, 17, 37 \}$.
\item If $\text{im}( \overline{\rho}_E^{\ell} ) \cong D_2$, then $\ell \in \{ 3, 5 \}$.
\item If $\text{im}( \overline{\rho}_E^{\ell} ) \cong D_n$ for some odd $n \geq 3$ and $\ell$ is odd, then $n = 3$ and $\ell = 7$.
\item If $\text{im}( \overline{\rho}_E^{\ell} ) \cong D_n$ for some even $n \geq 4$, then $(n, \ell) \in \{ (4, 3),  (4,5), (6,5), (6, 7), (8, 7), (12, 11) \}$.
\item If image of $\overline{\rho}_E^{\ell}$ is isomorphic to $S_4$, then $\ell \in \{ 5, 13 \}$.
\end{itemize}  

We will discuss examples of all these cases below. 

\begin{opm}\label{mod2opm}
Because $\text{GL}_2( \mathbb{F}_{2} ) \cong S_3$, by Theorem \ref{main2}, even in the case of full mod 2 image, the rational newform $f$ will be weakly $2$-abelian for the class $0 \bmod 2$. In this setting of elliptic curves, we can also give a more direct proof. If $E / \mathbb{Q}$ is an elliptic curve with minimal discriminant $\Delta$ and conductor $N$, then for $p > 2$ of good reduction, the definition $a_p(E) = p + 1 - \# E( \mathbb{F}_p )$ shows that $a_p(E) \equiv 0 \bmod 2$ if and only if $\# E( \mathbb{F}_p ) \equiv 0 \bmod 2$. This is equivalent to the finite group $E( \mathbb{F}_p )$ admitting a 2-torsion point. If a minimal model for $E$ is given by $y^2 = g(x) = x^3 + ax + b$ for some $a,b \in \mathbb{Z}$, then the $x$-coordinates of the 2-torsion points of $E$ are given by the roots of $g$. The discriminant of $g$ equals $\Delta$ up to squares, and the splitting field of $g$ will contain the field $\mathbb{Q}( \sqrt{\Delta} )$. If $\Delta$ is not a square modulo $p$, this yields a quadratic subfield of the splitting field of $g$ over $\mathbb{F}_p$. Therefore this splitting field needs to be quadratic as well, and $g$ must have a root in $\mathbb{F}_p$, yielding a 2-torsion point in $E( \mathbb{F}_p )$. In other words, we have shown that
\[
\left( \frac{\Delta}{p} \right) = -1 \implies a_p(E) \equiv 0 \bmod 2.
\]
\end{opm}

\begin{vb}\label{338exam}
We consider the rational newform of weight 2 and level $338 = 2 \cdot 13^2$ associated with the isogeny class of elliptic curves with Cremona label 338d. Its $\textbf{q}$-expansion starts with
\[
f = \textbf{q} -  \textbf{q}^{2} -  \textbf{q}^{3} +  \textbf{q}^{4} + 3 \textbf{q}^{5} +  \textbf{q}^{6} + 3 \textbf{q}^{7} -  \textbf{q}^{8} - 2 \textbf{q}^{9} - 3 \textbf{q}^{10} -  \textbf{q}^{12} - 3 \textbf{q}^{14} - 3 \textbf{q}^{15} +  \textbf{q}^{16} + \ldots
\]
We will study the behaviour of its Fourier coefficients modulo the three smallest primes $\ell \in \{ 2, 3, 5 \}$. First, we remark that the mod 2 image is full. We compute that $\mathbb{Q}( \sqrt{ \Delta } ) = \mathbb{Q}( \sqrt{-26} )$, and so by Remark \ref{mod2opm}, 
\[
\left( \frac{-26}{p} \right) = -1 \implies a_p(E) \equiv 0 \bmod 2.
\]
Next, the mod 3 images of the Galois representations on the 3-adic Tate modules of the elliptic curves are isomorphic to $D_4 \subset \text{GL}_2( \mathbb{F}_3 )$, with projective image isomorphic to $D_2 \cong C_2 \times C_2$. From Theorem \ref{main3} and Proposition \ref{mod}, we expect the primes $p$ for which $a_p(f) \equiv 0 \bmod 3$ to be given by a congruence condition on $p$ with the modulus dividing $\text{rad}( 338 \cdot 3 ) \cdot \text{gcd}(2, 338 )^2 = 2^3 \cdot 3 \cdot 13$. The extension governing the vanishing of $a_p(f) \bmod 3$ is the field $\overline{Z}_3 = \mathbb{Q}( \sqrt{-3}, \sqrt{13} ) \subset \mathbb{Q}( \zeta_{39} )$, as this is the unique biquadratic subfield of the 3-division fields of the relevant elliptic curves. The vanishing of $a_p(f) \bmod 3$ is thus determined by the value of $p \bmod 39$. Explicitly,
\[
a_p(f) \equiv 0 \bmod 3 \iff \left( \frac{-3}{p} \right) = -1 \quad \text{or} \quad \left( \frac{13}{p} \right) = -1.
\]
The mod 5 image of any of these elliptic curves is the full Borel subgroup of $\text{GL}_2( \mathbb{F}_5 )$ of cardinality 80. Even though this group is nonabelian, the image of its semisimplification is the maximal diagonal subgroup of $\text{GL}_2( \mathbb{F}_5 )$, which is an abelian group isomorphic to $C_4 \times C_4$. Theorem \ref{main1} predicts the value of $a_p(f) \bmod 5$ to be given by a congruence condition on $p$, and according to Proposition \ref{mod} with the modulus being a divisor of $\text{rad}( 338 \cdot 5 ) \cdot \text{gcd}(8, \{ 2, 5, 13 \} ) = 2^4 \cdot 5 \cdot 13$. To start, one may observe that for $p \neq 2, 5, 13$,
\[
a_p(f) \bmod 5 \equiv
\begin{cases}
0 & \implies p \equiv 1,4 \bmod 5 \\
1,4 &\implies p \equiv 3, 4 \bmod 5; \\
2,3 &\implies p \equiv 1, 2 \bmod 5. \\
\end{cases}
\]
This can be explained from the fact that the determinant of the image of $\text{Frob}_{\mathfrak{p}}$ is equal to $p \bmod 5$. Namely, if $\rho_f^{\text{ss}, 5}( \text{Frob}_{\mathfrak{p}} )$ is the diagonal matrix with entries $\alpha, \beta \in \mathbb{F}_5$, then
\[
\alpha \beta \equiv p \bmod 5 \quad \text{and} \quad \alpha + \beta \equiv a_p \bmod 5 \implies p + \beta^2 \equiv a_p \beta \bmod 5. 
\]
Therefore $( 2\beta - a_p)^2 \equiv p + a_p^2 \bmod 5$. The observations now coincide with the simple criterion of $p + a_p^2$ being a square modulo 5. One may also analyse the situation modulo 65, and find that
\[
a_p(f) \bmod 5 \equiv
\begin{cases}
0 & \iff p \equiv 4, 6, 9, 11, 14, 21, 29, 31, 41, 46, 49, 64 \bmod 65; \\
1 &\iff p \equiv 8, 19, 23, 33, 38, 43, 44, 54, 63 \bmod 65; \\
2 &\iff p \equiv 1, 2, 12, 16, 17, 32, 57, 61, 62 \bmod 65; \\
3 &\iff p \equiv 7, 22, 27, 36, 37, 42, 47, 51, 56 \bmod 65; \\
4 &\iff p \equiv 3, 18, 24, 28, 34, 48, 53, 58, 59 \bmod 65. \\
\end{cases}
\]
These congruences are a shadow of the splitting behaviour of a rational prime in the abelian extension $Z_5$ of degree 16 cut out by the image of  $\rho_f^{\text{ss}, 5}$. All elliptic curves in the class attain a 5-torsion point over a $C_4$-field that is unramified away from $5$ and $13$ and thus contained in $\mathbb{Q}( \zeta_{65} )$. The top left entry of $\rho_f^{\lambda}( \text{Frob}_{\mathfrak{p}} )$ is now fully determined by $p \bmod 65$, and since the determinant equals $p \bmod 5$, the trace is also determined by $p \bmod 65$. We now identify $Z_5$ as the unique subfield of degree 16 contained in the degree 48 field $\mathbb{Q}( \zeta_{65} )$. The sizes of the congruence classes are explained by the observation that the maximal diagonal subgroup of $\text{GL}_2( \mathbb{F}_5 )$ contains 4 traceless matrices, and for each non-zero value of the trace only 3 elements.
\end{vb}

\begin{vb}\label{s3exam1}
We consider the rational newform of weight 2 and level $2450 = 2 \cdot 5^2 \cdot 7^2$ associated with the elliptic curve with Cremona label 2450ba1. Its $\textbf{q}$-expansion starts with
\[
f = \textbf{q} +  \textbf{q}^{2} +  \textbf{q}^{4} +  \textbf{q}^{8} - 3\textbf{q}^{9} - 2\textbf{q}^{11} +  \textbf{q}^{16} - 7 \textbf{q}^{17} - 3 \textbf{q}^{18} + \ldots.
\]
The mod 7 image of the Galois representation on the 7-adic Tate module of the elliptic curve is abstractly isomorphic to $C_3 \times S_3$ with $C_3$ as its center. Its projective image is therefore isomorphic to $S_3 = D_3$. By Theorem \ref{main3}, this modular form should be $7$-abelian for the class $0$, and by Theorem \ref{main4} the other classes should be semi-$7$-abelian. Experimentally, one finds that
\[
a_p(f) \bmod 7 \equiv
\begin{cases}
0 &\text{if $p \equiv 0, 3, 5, 6 \bmod 7$;} \\
1, 3 &\text{if $p \equiv 2 \bmod 7$;} \\
2, 6 &\text{if $p \equiv 1 \bmod 7$;} \\
4, 5 &\text{if $p \equiv 4 \bmod 7$.} \\
\end{cases}
\]
As promised by Proposition \ref{goodmod}, the primes for which $a_p \equiv 0 \bmod 7$ are those for which $p$ is not a square modulo $7$. We can explain the precise congruences exhibiting the semi-$7$-abelian nature of the other classes as follows. Explicitly, the image $G_7$ of $\rho_f^7$ is generated by the matrices
\[
\begin{pmatrix} 0 & 1 \\ 1 & 0 \end{pmatrix} \quad \text{and} \quad \begin{pmatrix} 2 & 0 \\ 0 & 1 \end{pmatrix} \quad \text{in} \quad \text{GL}_2( \mathbb{F}_7 ), \quad \text{with $[G_7, G_7]$ generated by} \quad \begin{pmatrix} 2 & 0 \\ 0 & 4 \end{pmatrix}.
\]
The quotient of $G_7$ by $[G_7, G_7]$ is isomorphic to $C_6$ and corresponds with the abelian extension $Z_7 = \mathbb{Q}( \zeta_7 )$, which by virtue of the Weil pairing will always be contained in the smallest field $L_7$ over which the elliptic curve obtains its full 7-torsion. This yields an alternative explanation for the conclusion of Proposition \ref{goodmod} in this setting. The six cosets can now be represented by $\text{id}$, $2 \cdot \text{id}$, $4 \cdot \text{id}$ and further three matrices with zeroes on their diagonals. The latter three cosets only contain traceless matrices; this again recovers the $7$-abelian nature of the vanishing of $a_p(f) \bmod 7$. The first coset only contains matrices of trace $2$ and $6$, and the remaining ones only of traces $4$ and $5$, or $1$ and $3$ respectively.

As a slight variant on the example above, we may also consider the newform associated with the curve with Cremona label 2450a1, whose $\textbf{q}$-expansion reads
\[
f = \textbf{q} -  \textbf{q}^{2} +  \textbf{q}^{4} -  \textbf{q}^{8} - \textbf{q}^{9} - \textbf{q}^{11} +  \textbf{q}^{16} + 7 \textbf{q}^{17} + 3 \textbf{q}^{18} + \ldots.
\]
This time, the associated mod 7 image is of cardinality 36 and abstractly isomorphic to $C_6 \times S_3$. Its projective image is isomorphic to $S_3= D_3$ once more, but the value of $p \bmod 7$ provides less information than before. Considering also the prime $5$ now seems to suggest the following additional congruences:
\[
a_p(f) \bmod 7 \equiv
\begin{cases}
1 &\implies p \equiv 8, 9, 16, 22 \bmod 35; \\
2 &\implies p \equiv 1, 18, 29, 32 \bmod 35; \\
3 &\implies p \equiv 9, 18, 32, 16 \bmod 35; \\
\end{cases} \qquad 
\begin{cases}
4 &\implies p \equiv 2, 4, 11, 23 \bmod 35; \\
5 &\implies p \equiv 4, 8, 11, 22 \bmod 35; \\
6 &\implies p \equiv 1, 2, 23, 29 \bmod 35. \\
\end{cases}
\]
The explanation proceeds as before, but is a bit more laborious. The derived subgroup of the image $G_7$ is generated by the same matrix as above, but now its quotient is isomorphic to $C_6 \times C_2$. One checks that the relevant abelian extension is given by $Z_7 = \mathbb{Q}( \zeta_7, \sqrt{5} ) / \mathbb{Q}$. This field is contained in $\mathbb{Q}( \zeta_{35} )$, and so the splitting behaviour of a rational prime is determined by its value modulo 35. 

Again, half the cosets of the commutator subgroup will consist solely of traceless matrices. In addition, the primes that split completely in $\mathbb{Q}( \zeta_7, \sqrt{5} )$ must be either 1 of 29 modulo 35. This trivial coset still admits only the traces 2 and 6; we leave it to the reader to work out the remaining five cosets.  
\end{vb}

\begin{vb}\label{counter}
We continue by considering the rational newform of weight 2 and level $608 = 2^5 \cdot 19$ associated with the elliptic curve with Cremona label 608e1. Its $\textbf{q}$-expansion starts with
\[
f = \textbf{q} + 3\textbf{q}^5 - 5\textbf{q}^7 - 3\textbf{q}^9 - 5\textbf{q}^{11} - 4\textbf{q}^{13} - 3\textbf{q}^{17} + \textbf{q}^{19} + \ldots
\]
The projective mod 5 image of the Galois representation on the 5-adic Tate module of the elliptic curve is isomorphic to $D_4$. In this case, there is a $5$-torsion point that generates a degree 8 field containing $\mathbb{Q}(i)$. Since $\mathbb{Q}( \zeta_5 ) \subset Z_5$, we must have $\mathbb{Q}( \sqrt{5} ) \subset \overline{Z}_5$, which implies that the quotient $D_4 / [D_4, D_4]$ corresponds to the biquadratic field $\overline{Z}_5 = \mathbb{Q}(i, \sqrt{5}) \subset \mathbb{Q}( \zeta_{20} )$. Two cosets consist of two involutions, one coset of an involution and the trivial element, and the final coset contains two elements of order 4. One now finds that
\[
p \equiv 3 \bmod 4 \implies a_p(f) \equiv 0 \bmod 5,
\]
whereas the converse does \emph{not} hold, for when $p \in \{ 1, 9 \} \bmod 20$, both zero and non-zero values of $a_p(f) \bmod 5$ are possible. In other words, the cyclic subgroup of rotations inside $D_4$ corresponds to $\mathbb{Q}(i)$, and not $\mathbb{Q}( \sqrt{ 5 } )$, which Proposition \ref{goodmod} predicts should always happen when $n$ is odd. This also illustrates that $f$ is \emph{weakly} $5$-abelian for the class $0$, but it is not $5$-abelian.
\end{vb}

\begin{vb}\label{S4exam1}
We next consider the rational newform of weight 2 and level $324 = 2^2 \cdot 3^4$ associated with the elliptic curve with Cremona label 324b1. Its $\textbf{q}$-expansion starts with
\[
f = \textbf{q} + 3 \textbf{q}^5 + 2 \textbf{q}^7 - 6 \textbf{q}^{11} + 5 \textbf{q}^{13} - 3 \textbf{q}^{17} + 2 \textbf{q}^{19} + 6 \textbf{q}^{23} + 4 \textbf{q}^{25} + 3 \textbf{q}^{29} - 4 \textbf{q}^{31} + 6 \textbf{q}^{35} + 5 \textbf{q}^{37} + \ldots
\]
The mod 5 image of the Galois representation on the 5-adic Tate module of the elliptic curve is of cardinality 96 and abstractly isomorphic to $U_2( \mathbb{F}_3 )$, where $U_2( \mathbb{F}_3 )$ denotes the unitary group of $\mathbb{F}_3^2$. Its center is of size 4 and its projective image is of size 24 and isomorphic to $S_4$. 

In the proof of Theorem \ref{main4}, we saw that the commutator subgroup $[G_5,G_5]$ is of size 24, and so of index 4 in $G_5$. In addition, Lemma \ref{trlem3} shows that on this subgroup all traces should occur. The determinant map $\text{det} : G_5 \to \mathbb{F}_{5}^{\times}$ is surjective and so the cosets of $[G_5,G_5]$ in $G_5$ are sets of matrices with some fixed determinant. This shows that the coset of $[G_5,G_5]$ that the element $\text{Frob}_{\mathfrak{p}}$ belongs to is fully determined by the value of $p \bmod 5$. In addition, if the determinant is $\pm 1 \bmod 5$, then we may take a scalar representative of the coset, and it again follows that all traces occur. We saw in the proof of Theorem \ref{main4} that any coset lifting $S_4 \setminus A_4$ in $G_5$ can contain at most 3 traces. Doing the computation, we find that for $p \geq 7$,
\[
a_p(f) \equiv 1, 4 \bmod 5 \implies p \not\equiv 2 \bmod 5 \quad \text{and} \quad a_p(f) \equiv 2, 3 \bmod 5 \implies p \not\equiv 3 \bmod 5.
\]
This illustrates that $f$ is semi-$5$-abelian for all classes in $\mathbb{F}_5^{\times}$, as predicted by Theorem \ref{main4}.
\end{vb}

\begin{vb}\label{S4exam2}
We conclude by considering the rational newform of weight 2 and level $50700 = 2^2 \cdot 3^2 \cdot 5 \cdot 13^2$ associated with the famous elliptic curve with Cremona label 50700u1. Its $\textbf{q}$-expansion starts with
\[
f = \textbf{q} + \textbf{q}^3 + \textbf{q}^9 - 3\textbf{q}^{11} + 7\textbf{q}^{23} + \textbf{q}^{27} - 4 \textbf{q}^{29} - 3\textbf{q}^{33} - \textbf{q}^{37} + \ldots
\]
The mod 13 image of the Galois representation on the 13-adic Tate module of the elliptic curve is of cardinality 288 and abstractly isomorphic to $C_3 \times U_2( \mathbb{F}_3 )$, so its projective image is also isomorphic to $S_4$. By Lemma \ref{trlem3}, on $[G_{13}, G_{13}]$ only the traces $0$, $\pm 1$ and $\pm 2$ should occur. As before, by taking a scalar coset representative, it follows that the only possible traces for $p \equiv x^2 \bmod 13$ are $0, \pm x, \pm 2x \bmod 13$. A similar connection holds between the three possible traces for non-square determinants. We find for $p \neq 2, 3, 5, 13$,
\[
p \bmod 13 =
\begin{cases}
1 & \implies a_p(f) \equiv -2, -1, 0, 1, 2 \bmod 13; \\
2 & \implies a_p(f) \equiv -2, 0, 2 \bmod 13; \\
3 & \implies a_p(f) \equiv -5, -4, 0, 4, 5 \bmod 13; \\
4 & \implies a_p(f) \equiv -4, -2, 0, 2, 4 \bmod 13; \\
5 & \implies a_p(f) \equiv -6, 0, 6 \bmod 13; \\
6 & \implies a_p(f) \equiv -5, 0, 5 \bmod 13;
\end{cases}
\qquad 
\begin{cases}
7 & \implies a_p(f) \equiv -1, 0, 1 \bmod 13; \\
8 & \implies a_p(f) \equiv -4, 0, 4 \bmod 13; \\
9 & \implies a_p(f) \equiv -6, -3, 0, 3, 6 \bmod 13; \\
10 & \implies a_p(f) \equiv -6, -1, 0, 1, 6 \bmod 13; \\
11 & \implies a_p(f) \equiv -3, 0, 3 \bmod 13; \\
12 & \implies a_p(f) \equiv -5, -3, 0, 3, 5 \bmod 13.
\end{cases}
\] 
In spite of these remarkable congruences, it is seen that $0 \in \mathbb{F}_{13}$ is in fact not semi-$\lambda$-abelian.
\end{vb}

\subsection*{Acknowledgements}

The author thanks Pieter Moree for raising the question that is being addressed in this work during a tea break at the Max Planck Institute for Mathematics in Bonn. In addition, the author is indebted to Pieter Moree, Diana Mocanu and Steven Charlton for reading an earlier version of this manuscript and providing helpful feedback. The author also thanks Diana Mocanu and Pengcheng Zhang for helpful discussions and comments. The author is indebted to Robert Pollack for help with computational aspects, the results of which did not find their way into the final version of this work.

\bibliographystyle{alpha}
\bibliography{abcongs}

\end{document}